\documentclass[12pt]{article}
\usepackage{graphicx}
\usepackage{amsmath}
\usepackage{amssymb}
\usepackage{theorem}
\usepackage{fancyhdr}
\usepackage{hyperref}

\numberwithin{equation}{section}

\newtheorem{thm}{Theorem}[section]
\newtheorem{lemma}[thm]{Lemma}
\newtheorem{prop}[thm]{Proposition}
\newtheorem{cor}[thm]{Corollary}
{\theorembodyfont{\rmfamily}
\newtheorem{defn}[thm]{Definition}
\newtheorem{eg}[thm]{Example}

\newtheorem{rmk}[thm]{Remark}
}

\newcommand{\qed}{\hfill \mbox{\raggedright \rule{.07in}{.1in}}}
 
\newenvironment{proof}{\vspace{1ex}\noindent{\bf
Proof}\hspace{0.5em}}{\hfill\qed\vspace{1ex}}
\newenvironment{pfof}[1]{\vspace{1ex}\noindent{\bf Proof of
#1}\hspace{0.5em}}{\hfill\qed\vspace{1ex}}

\newcommand{\R}{{\mathbb R}}

\newcommand{\C}{{\mathbb C}}
\newcommand{\Z}{{\mathbb Z}}
\newcommand{\D}{{\mathbb D}}
\newcommand{\T}{{\mathbb T}}
\newcommand{\N}{{\mathbb N}}

\newcommand{\cB}{{\mathcal{B}}}
\newcommand{\cX}{{\mathcal{X}}}

 \newcommand{\Fix}{\operatorname{Fix}}
 \newcommand{\eps}{{\epsilon}}
 \newcommand{\spec}{\operatorname{spec}}
 \newcommand{\diam}{\operatorname{diam}}
 \newcommand{\supp}{\operatorname{supp}}
 \newcommand{\dist}{\operatorname{dist}}

\newcommand{\SMALL}{\textstyle}
\newcommand{\BIG}{\displaystyle}

\newcommand{\vertiii}[1]{{\left\vert\kern-0.25ex\left\vert\kern-0.25ex\left\vert #1 
    \right\vert\kern-0.25ex\right\vert\kern-0.25ex\right\vert}}

\title{Mixing Properties for Toral Extensions \\ 
of Slowly Mixing Dynamical Systems \\ with Finite and Infinite Measure} 

\author{
Ian Melbourne \thanks{Mathematics Institute, University of Warwick, Coventry, CV4 7AL, UK}
\and 
Dalia Terhesiu
\thanks{Faculty of Mathematics, University of Vienna,
Oskar Morgensternplatz 1, 1090 Vienna, AUSTRIA.
Current address: Mathematics Department, University of Exeter, EX4 4QF, UK}
}

\date{30 November 2015.  Updated 25 October 2018.}

\begin{document}

\maketitle

\begin{abstract}
We prove results on mixing and mixing rates for toral extensions of nonuniformly expanding maps with subexponential decay of correlations.  Both the finite and infinite measure settings are considered.  Under a Dolgopyat-type condition on nonexistence of approximate eigenfunctions, we prove that existing results for (possibly nonMarkovian) nonuniformly expanding maps hold also for their toral extensions.
 \end{abstract}

 \section{Introduction} 
 \label{sec-intro}

In a landmark paper, Dolgopyat~\cite{Dolgopyat02} obtained results on superpolynomial decay of correlations for compact group extensions of uniformly expanding and uniformly hyperbolic dynamical systems.  
An interesting question is to extend this result to nonuniformly expanding/hyperbolic systems, including systems that are slowly mixing or preserving an infinite measure.

In this paper, we focus on the case when the group is abelian, and consider 
toral extensions of a large class of (not necessarily Markov) nonuniformly expanding maps, including the AFN maps of~\cite{Zweimuller98,Zweimuller00}, in both the finite and infinite measure settings.
Under mild hypotheses, we show that sharp mixing results for the underlying map pass over to the toral extension.

Future projects could include group extensions of nonuniformly hyperbolic systems including the case of general compact groups.  
Passing from nonuniformly expanding to nonuniformly hyperbolic should be straightforward for systems with exponentially contracting stable directions, but this is a somewhat restrictive assumption.  For recent substantial progress on the analogous question for nonuniformly hyperbolic flows and comparison with the group extension situation, see~\cite{BBMsub} and~\cite[Section~9]{rapid} respectively.

The analysis of compact group extensions divides into the cases where the group is abelian (a torus) or semisimple.  As seen in~\cite{Dolgopyat02} (see also~\cite{FieldParry99}), it turns out that the toral case raises more technical difficulties, though the semisimple case is more complicated in terms of notation and prerequisites from representation theory.
In this paper, we have chosen to focus on the technically harder toral case; we do not anticipate any major difficulties in dealing with general compact groups but have not investigated this further.

\subsection{Existing results for nonuniformly hyperbolic maps}

Let $(X,d)$ be a metric space with Borel measure $\mu$,
and let $f:X\to X$ be an ergodic and topologically mixing measure-preserving transformation.
Let $Y\subset X$ be a subset with $\mu(Y)\in(0,\infty)$.
We define the first return time
$\tau:Y\to\Z^+$ and first return map $F=f^\tau:Y\to Y$ given by
\[
	\tau(y)=\inf\{n\ge1:f^ny\in Y\} \quad\text{and}\quad
F(y)=f^{\tau(y)}(y).
\]
Under certain assumptions on $F$ and $\tau$,
it is possible to obtain sharp mixing properties for $f$.
More specifically, we assume that
\begin{itemize}
\item[(i)] The first return time $\tau:Y\to\Z^+$ is either nonintegrable
with $\mu(y\in Y:\tau(y)>n)=\ell(n)n^{-\beta}$ where $\beta\in(0,1]$ and
$\ell$ is a slowly varying function\footnote{A measurable function $\ell:(0,\infty)\to(0,\infty)$ is {\em slowly varying} if $\lim_{x\to\infty}\ell(\lambda x)/\ell(x)=1$ for all $\lambda>0$.},
or integrable with $\mu(y\in Y:\tau(y)>n)=O(n^{-\beta})$ where $\beta>1$.
\item[(ii)] The first return map $F:Y\to Y$ fits into the appropriate functional abstract framework with suitable Banach space of observables $\cB(Y)\subset L^1(Y)$ with norm $\|\;\|$
(see~\cite{Gouezel04a,Sarig02} for the finite measure case, and~\cite{MT12} for the infinite measure case).
\end{itemize}

Under conditions (i) and (ii), we recall the following results from~\cite{Gouezel04a,Sarig02} and~\cite{MT12} for the map $f:X\to X$
and observables $v_0$, $w_0$ supported in $Y$ with $v_0\in\cB(Y)$, $w_0\in L^\infty(Y)$.  
Let $\bar v_0=\int_Y v_0\,d\mu$, $\bar w_0=\int_Y w_0\,d\mu$.

In the infinite measure case, define
\begin{align} \label{eq-ell}
\tilde\ell(n)=\begin{cases} \ell(n), & \beta\in(0,1) \\
\sum_{j=1}^n\ell(j)j^{-1}, & \beta=1 \end{cases},
	\quad\text{and}\quad
d_\beta=\begin{cases} \frac{1}{\pi}\sin\beta\pi, & \beta\in(0,1) \\
1, & \beta=1\end{cases}.
\end{align}
If $\beta\in(\frac12,1]$, then
\begin{align} \label{eq-inf1}
	\lim_{n\to\infty}\tilde\ell(n)n^{1-\beta}\int_Y v_0\,w_0\circ f^n\,d\mu= d_\beta \bar v_0\bar w_0.
\end{align}
If $\beta<\frac12$, or if $\beta<1$ and either 
$\bar v_0=0$ or $\bar w_0=0$, then
\begin{align} \label{eq-inf2}
\int_Y v_0\,w_0\circ f^n\,d\mu=O(\ell(n)n^{-\beta}\|v_0\||w_0|_\infty).
\end{align}

In the finite measure case, we normalise so that $\mu$ is a probability measure on $X$.
For all $n\ge1$, 
\begin{align} \label{eq-fin}
\int_Y v_0\,w_0\circ f^n\,d\mu
- \bar v_0\bar w_0
= \sum_{j>n}\mu(\tau>j) \bar v_0\bar w_0 +E_\beta(n)\|v_0\||w_0|_\infty,
\end{align}
where $E_\beta(n)=O(n^{-\beta})$ for $\beta>2$, 
$E_\beta(n)=O(n^{-2}\log n)$ for $\beta=2$, 
and $E_\beta(n)=O(n^{-(2\beta-2)})$ for $1<\beta<2$.
Also $E_\beta(n)=O(n^{-\beta})$ for all $\beta>1$ if 
 $\bar v_0=0$ or $\bar w_0=0$.

 \begin{rmk}  The precise functional analytic hypotheses mentioned in condition~(ii) play no role in this paper; we use only the 
consequences~\eqref{eq-inf1}--\eqref{eq-fin} for $\int_Y v_0\,w_0\circ f^n\,d\mu$.
	A special case is when $F$ is a full branch Gibbs-Markov map 
	with $\cB(Y)$ taken to be a space $F_\theta(Y)$ of Lipschitz observables  (see Subsection~\ref{sec-toral} and Section~\ref{sec-GM} for definitions).  For nonMarkov examples, see Subsection~\ref{sec-eg}.
\end{rmk}

\subsection{Toral extensions}
\label{sec-toral}

\paragraph{Set up}
In this paper, we prove analogous results for toral extensions of nonuniformly expanding maps $f:X\to X$ satisfying conditions~(i) and~(ii).
We assume further that there exists $Z\subset Y\subset X$ (possibly $Z=Y$) with $\mu(Z)>0$ and a return time\footnote{
A function $\varphi:Z\to\Z^+$ is called a {\em return time} if 
$f^{\varphi(z)}z\in Z$ for all $z\in Z$.}
 $\varphi:Z\to \Z^+$ (not necessarily a first return time).  Define the return map $G=f^\varphi:Z\to Z$, $G(z)=f^{\varphi(z)}z$.  We assume:
\begin{itemize}
\item[(iii)]  
there is a measure $\mu_Z$ on $Z$ equivalent to $\mu|_Z$ and
an at most countable measurable partition $\alpha$ of $Z$ such that
$\varphi$ is constant on partition elements and
$\gcd\{\varphi(a):a\in\alpha\}=1$.
Moreover, there are constants $\lambda>1$, $\eta\in(0,1]$, $C_1\ge1$, such that for each $a\in\alpha$,
\begin{itemize}
\item[(1)] $G:a\to Z$ is a measure-theoretic bijection.
\item[(2)] $d(Gz,Gz')\ge \lambda d(z,z')$ for all $z,z'\in a$.
\item[(3)] $g=\log \frac{d\mu_Z}{d\mu_Z\circ G}$
satisfies $|g(z)-g(z')|\le C_1d(Gz,Gz')^\eta$ for all
\mbox{$z,z'\in a$}.
\item[(4)] $d(f^\ell z,f^\ell z')\le C_1d(Gz,Gz')$ for all $z,z'\in a$,
$0\le \ell <\varphi(a)$.
\end{itemize}
\end{itemize}

In particular, conditions (1)--(3) mean that 
$G:Z\to Z$ is a full branch Gibbs-Markov map with partition $\alpha$.  
Such maps are discussed further in Section~\ref{sec-GM}.

\begin{itemize}
	\item[(iv)] 
There exists $\rho:Z\to \Z^+$ constant on elements of the partition $\alpha$ such that $G(z)=F^{\rho(z)}z$ for $z\in Z$.  Moreover,
$\mu_Z(z\in Z: \rho(z)>n)=O(e^{-cn})$ for some $c>0$, and
if $a\in\alpha$, then $\tau\circ F^j$ is constant on $a$ for all
$j< \rho(a)$.
\end{itemize}
Assumptions similar to~(iv) were considered in~\cite{BruinTerhesiu18}.

\begin{rmk}  \label{rmk:beta}
	Note that $\varphi=\tau_\rho:Z\to\Z^+$,
\[
	\varphi(z)=\tau_{\rho(z)}(z)={\SMALL\sum_{j=0}^{\rho(z)-1}}\tau\circ F^j.
\]
It follows from assumptions~(i) and~(iv) by an
elementary calculation~\cite{Markarian04} (see also~\cite[Theorem~4]{ChernovZhang05}) that
$\mu_Z(\varphi>n)=O(n^{-\beta'})$ for any specified $\beta'<\beta$.
Moreover, it suffices in (iv) that $\mu_Z(\rho>n)=O(n^{-q})$ for $q$ sufficiently large.

In certain situations, including the examples in Subsection~\ref{sec-eg}, it is possible to achieve $\beta'=\beta$.  However, this does not lead to improvements in our main results, so we generally ignore this possibility.  (On the other hand, the upper bound result Corollary~\ref{cor-upper} does depend on the specific decay rate for $\mu_Z(\varphi>n)$.)
\end{rmk}

Let $h:X\to \T^d$ be a measurable map; following standard conventions we refer to $h$ as a {\em cocycle}.  We assume that $h$ is $C^\eta$.
(More precisely, view $\T^d$ as a compact group of diagonal $d\times d$ complex matrices with distance $|\;|$.  We require that $|h|_\eta=\sup_{x\neq x'}|h(x)-h(x')|/d(x,x')^\eta<\infty$.)
 Form the {\em toral extension}
\[
f_h:X\times\T^d\to X\times\T^d, \quad
f_h(x,\psi)=(fx,\psi+h(x)).
\]
The product measure $m=\mu\times d\psi$ is $f_h$-invariant.


 It is necessary to rule out certain pathological cases, since toral extensions of mixing uniformly expanding maps need not be mixing, and mixing toral extensions can mix arbitrarily slowly. Dolgopyat~\cite{Dolgopyat98b,Dolgopyat02} introduced  condition (v) below for proving superpolynomial decay of correlations for suspensions and compact group extensions of uniformly expanding/hyperbolic systems.  
Our final assumption is
\begin{itemize} 
\item[(v)] There do not exist approximate eigenfunctions.
\end{itemize}

The definition of approximate eigenfunctions is somewhat technical, and so is delayed until Section~\ref{sec-eigen} where we
show that condition~(v) holds typically in a strong probabilistic sense.
In Appendix~\ref{app-good}, we show that condition~(v) holds for an open and dense set of smooth toral extensions.

\subsection*{Mixing results for toral extensions}

Let $f:X\to X$ be a topologically mixing map with ergodic invariant measure $\mu$ 
and $h:X\to\T^d$ be a $C^\eta$ cocycle, $\eta\in(0,1]$. We consider 
toral extensions $f_h:X\times\T^d\to X\times\T^d$ as described previously satisfying conditions~(i)--(v), where condition~(ii) can be replaced by the fact that~\eqref{eq-inf1}--\eqref{eq-fin} hold for observables $v_0\in\cB(Y)$ and $w_0\in L^\infty(Y)$.

Let $v:X\times\T^d\to\R$.  For $\eta\in(0,1]$,
define $|v|_{C^\eta}=\sup_{\psi\in\T^d}\sup_{x\neq y}|v(x,\psi)-v(y,\psi)|/d(x,y)^\eta$ and $\|v\|_{C^\eta}=|v|_\infty+|v|_{C^\eta}$.
Write $v\in C^\eta(X\times\T^d)$ if $\|v\|_{C^\eta}<\infty$.

For $\eta\in(0,1]$ and $p\in\N$,
write
$v\in C^{\eta,p}(X\times\T^d)$ if $v$ is $p$-times differentiable
with respect to~$\psi$ with derivatives that lie in $C^\eta(X\times\T^d)$,
and  set
$\|v\|_{C^{\eta,p}}=\sum_{|j|\le p}
\|\frac{\partial^jv}{\partial\psi^j}\|_{C^\eta}$.\footnote{
Given $j\in\Z^d$ with $j_1,\dots,j_d\ge0$, we write
$|j|=j_1+\dots+j_d$ and 
$\BIG\frac{\partial^j}{\partial\psi^j}=
\frac{\partial^{|j|}}{\partial\psi_1^{j_1}\cdots\partial\psi_d^{j_d}}$. }

For our main results, we consider observables $v,w$ supported in $Y\times\T^d$.
Let $v_0(y)=\int_{\T^d}v(y,\psi)\,d\psi$.
Suppose that
$v_0\in\cB(Y)$, $v-v_0\in C^{\eta,p}(Y\times\T^d)$, $w\in L^\infty(Y\times\T^d)$,
where $p\in\N$ is chosen sufficiently large (depending only on $\eta$, $d$, and the measure $\mu$ on $X$),
and write $\vertiii{v}=\|v_0\|+\|v-v_0\|_{C^{\eta,p}}$.
Let $\bar v=\int_{Y\times\T^d}v\,dm$,
$\bar w=\int_{Y\times\T^d}w\,dm$.

\begin{thm} \label{thm-infinite}
In the infinite measure case, define
$\tilde\ell$ and $d_\beta$ as in~\eqref{eq-ell}.
\begin{itemize}
\item[(a)] Suppose that $\beta\in(\frac12,1]$. Then
\[
\lim_{n\to\infty}
\tilde\ell(n)n^{1-\beta}\int_{Y\times\T^d}v\, w\circ f_h^n \,dm
=d_\beta \bar v\bar w.
\]
\item[(b)] Suppose either that $\beta\in(0,\frac12]$, or that 
$\beta\in(0,1]$ and either $\bar v=0$ or $\bar w=0$.  Then for all $\eps>0$,
\[
	\int_{Y\times\T^d}v\, w\circ f_h^n \,dm
	= O(n^{-(\beta-\eps)}\vertiii{v}|w|_\infty).
\]
\end{itemize}
\end{thm}

\begin{rmk} \label{rmk-error}
 Under stronger conditions on $\mu(\tau>n)$, improved error rates and higher order asymptotics are obtained for nonuniformly expanding maps $f$ in~\cite{MT12,Terhesiu15}.  This applies in particular to the Markov intermittent maps considered in~\cite{LiveraniSaussolVaienti99} and to the nonMarkov examples in Subsection~\ref{sec-eg} (for the nonMarkov examples, the stronger conditions on $\mu(\tau>n)$ are proved in~\cite{BruinTerhesiu18} as described in Subsection~\ref{sec-eg}).
The results in this paper show that these higher order results apply also to typical toral extensions of these intermittent maps.
\end{rmk}

\begin{thm}  \label{thm-finite}
In the finite measure case, for all $\eps>0$, 
\[
	\int_{Y\times\T^d}v\, w\circ f_h^n \,dm - \bar v\bar w
	= \sum_{j>n}\mu(\tau>j) \bar v\bar w 
	+O(n^{-q}\vertiii{v}\, |w|_\infty),
\]
where $q=\beta-\eps$ if $\beta\ge2$ and $q=2\beta-2$ if $1<\beta<2$.
We can also take $q=\beta-\eps$ if 
$\beta>1$ and $\bar v=0$ or $\bar w=0$.
\end{thm}

\paragraph{Strategy of the proofs}

For $L^2$ observables $v,w:X\times\T^d\to\R$, we write
\begin{align} \label{eq-vfourier}
	v(x,\psi)={\SMALL\sum_{k\in\Z^d}}v_k(x)e^{ik\cdot\psi},
\end{align}
where $v_k\in L^2(X,\C)$,\footnote{Since $v$ is real-valued, necessarily $v_{-k}$ is the complex conjugate of $v_k$.}
and similarly for $w$.
Conditions~(i) and (ii) above on the first return map $F=f^\tau:Y\to Y$ take care of the zero Fourier modes $v_0$ and
$w_0$, so the main contribution of the current paper is to deal with the nonzero modes.
In Section~\ref{sec-new}, we show how this can be achieved 
under conditions~(iii)--(v) using the induced map $G=f^\varphi:Z\to Z$.

\begin{rmk}  If the first return map $F=f^\tau:Y\to Y$ is a full branch Gibbs-Markov map, then there is no need for a second inducing scheme: we can simply take $G=F$.  (Conditions~(iii) and (iv) can now be ignored.)
Even here our results are new.
This simplified set up applies to the maps in Examples~\ref{eg-1} and~\ref{eg-2} below if they are Markov, and more generally to the class of Thaler maps~\cite{Thaler80}.

For the nonMarkovian ``AFN'' maps of~\cite{Zweimuller98,Zweimuller00}, we use both of the inducing schemes and our main theorems apply with $\cB(Y)$ taken to be the space of bounded variation functions on $Y$. 
This includes all cases in Examples~\ref{eg-1} and~\ref{eg-2}.
\end{rmk}

\paragraph{Upper bounds on decay of correlations}  In the finite measure case, we also obtain an upper bound for decay of correlations, see Corollary~\ref{cor-upper}.  This is simpler than the other results mentioned here, and we need only to use one inducing scheme, $G=f^\varphi:Z\to Z$, satisfying condition (iii)
with $\beta>1$.
In particular, our result applies to toral extensions of maps modelled by Young towers with polynomial tails and summable decay of correlations~\cite{Young99}, and shows under condition~(v) that the toral extension $f_h$ mixes at the same rate as $f$.

\subsection{Examples}
\label{sec-eg}

Prototypical examples include 
Pomeau-Manneville intermittent maps of the unit interval~\cite{PomeauManneville80} such as the following:

\begin{eg} \label{eg-1}
$f(x)=\begin{cases} x(1+c_1^\gamma x^\gamma), & x\in[0,\frac12) \\
2x-1, & x\in[\frac12,1] \end{cases}$,
where $\gamma>0$, $c_1\in(0,2]$.
When $c_1=2$, the map $f$ is Markov and was
introduced in~\cite{LiveraniSaussolVaienti99}.
\end{eg}

\begin{eg} \label{eg-2}
$f(x)=x(1+c_2x^\gamma) \bmod1$,
where $\gamma>0$, $c_2>0$.
If $c_2$ is an integer, then $f$ is Markov and belongs to the class of maps studied by~\cite{Thaler80}.
\end{eg}

In general, the above maps $f$ are nonMarkovian and are examples of ``AFN maps''~\cite{Zweimuller98,Zweimuller00}.
For all $\gamma>0$, there is a unique (up to scaling) $\sigma$-finite invariant measure $\mu$ equivalent to Lebesgue and the measure is finite if and only if $\gamma<1$.  

We now describe how to verify assumptions (i)-(v) for these examples.
In Example~\ref{eg-1}, it is convenient to take $Y=[\frac12,1]$.
In Example~\ref{eg-2}, a convenient choice is to let $Y$ be the domain of the right-most branch.  
Let $\beta=1/\gamma$.
By the proof of~\cite[Lemma~9.1]{BruinTerhesiu18} and 
by~\cite[Lemma~9.2]{BruinTerhesiu18}, there are constants $c_1,c_2>0$ such that
\begin{align} \nonumber
 \mu_Z(\varphi>n)& =c_1n^{-\beta}+O(n^{-2\beta},n^{-(\beta+1)}\log n),
\\ 
\mu(\tau>n)& =c_2\mu_Z(\varphi>n)+O(n^{-(\beta+1)}),
\label{eq-rate}
\end{align}
so condition (i) is satisfied.   

Condition~(ii) holds with 
$\cB(Y)$ taken to be the space of bounded variation functions on $Y$
(see for example~\cite[Proposition~11.10]{MT12}) and conditions~(iii,iv) are verified in~\cite[Section~9]{BruinTerhesiu18}. 
Condition~(v) is satisfied for typical H\"older cocycles $h$, see Proposition~\ref{prop-typical}, and also for an open and dense set of smooth cocycles, see Appendix~\ref{app-good}.  

Hence our main results apply to typical toral extensions of nonMarkovian intermittent maps.  Since the estimates~\eqref{eq-rate} for $\mu(\varphi>n)$ include error terms, we can obtain error rates and higher order asymptotics in the infinite measure case $\gamma\ge1$ as indicated in Remark~\ref{rmk-error}.

\vspace{1ex}
The remainder of the paper is structured as follows.
In Section~\ref{sec-new}, we state results,
Theorems~\ref{thm-S1} and~\ref{thm-S2}, on the nonzero Fourier modes in~\eqref{eq-vfourier} and use these to prove the results from the introduction.
In Section~\ref{sec-GM}, we recall the definition and basic properties of the Gibbs-Markov induced map $G=f^\varphi$.  
In Section~\ref{sec-eigen}, we recall the notions of eigenfunctions and approximate eigenfunctions.
In Section~\ref{sec-F}, we recall some standard results about smoothness of Fourier series.  
In Section~\ref{sec-induced}, we obtain some estimates for twisted transfer operators corresponding to the induced dynamics on $Y$, and we derive a Dolgopyat-type estimate.  
In Section~\ref{sec-T}, we obtain estimates for certain associated renewal operators.
Theorems~\ref{thm-S1} and~\ref{thm-S2} are proved in Sections~\ref{sec-S1} and~\ref{sec-S2} respectively.

\vspace{-1ex}
\paragraph{Notation}
We use ``big O'' and $\ll$ notation interchangeably, writing
$a_n=O(b_n)$ or $a_n\ll b_n$  if there is a constant
$C>0$ such that $a_n\le Cb_n$ for all $n\ge1$.

\section{Reduction to the nonzero Fourier modes}
\label{sec-new}

In this section, we show how to reduce to dealing with the nonzero Fourier modes in~\eqref{eq-vfourier}.
First, we require the following basic expansion of $\int_{X\times\T^d}v\,w\circ f_h^n\,dm$.
Note that $f_h^n(x,\psi)=(f^nx,\psi+h_n(x))$
where $h_n=\sum_{j=0}^{n-1}h\circ f^j$.

\begin{prop} \label{prop-Fourier}
Let $v,w:X\times\T^d\to\R$ be $L^2$ observables with Fourier series as
in~\eqref{eq-vfourier}.   Then
	$\int_{X\times \T^d}v\,w\circ f_h^n\,dm=
\sum_{k\in\Z^d}\int_X e^{ik\cdot h_n}v_{-k}\; w_k\circ f^n \,d\mu$ for all $n\ge0$.
\end{prop}

\begin{proof}
Expanding into Fourier series, 
\begin{align*}
	& \int_{X\times \T^d}v\,w\circ f_h^n\,dm=
\sum_{j,k\in\Z^d} \int_{X\times\T^d}
v_j(x)e^{ij\cdot\psi}w_k(f^nx)e^{ik\cdot(\psi+h_n(x))}\,dm
\\ & \;  =\sum_{j,k\in\Z^d} \int_X v_j(x)w_k(f^nx)e^{ik\cdot h_n(x)}\,d\mu \int_{\T^d}e^{i(j+k)\cdot\psi}\,d\psi
=\sum_{k\in\Z^d} \int_X v_{-k}(x)w_k(f^nx)e^{ik\cdot h_n(x)},
\end{align*}
as required.
\end{proof}

The next two results concern the nonzero Fourier modes
\[
S_{v,w}(n)=\sum_{k\in\Z^d\setminus\{0\}}\int_X e^{ik\cdot h_n}v_{-k}\; w_k\circ f^n \,d\mu.
\]

\begin{thm} \label{thm-S1}
Assume that the induced map $G=f^\varphi:Z\to Z$ 
and the $C^\eta$ cocycle $h:X\to\T^d$ satisfy conditions~(iii)--(v).

Then there exists $p\in\N$ such that for all
observables $v,w$ supported in $Y\times\T^d$ with
$v\in C^{\eta,p}(Y\times\T^d)$, $w\in L^\infty(Y\times\T^d)$, and for all $\eps>0$,
\[
S_{v,w}(n)=O(n^{-(\beta-\eps)}\|v-v_0\|_{C^{\eta,p}}|w|_\infty).
\]
\end{thm}

\begin{thm} \label{thm-S2}
Let $h:X\to\T^d$ be a $C^\eta$ cocycle, $\eta\in(0,1]$, and assume nonexistence of approximate eigenfunctions.
Let $\varphi:Z\to\Z^+$ be a (general) return time such that 
$\mu_Z(\varphi>n)=O(n^{-\beta})$ where $\beta>1$, 
and $G=f^\varphi:Z\to Z$ is full branch Gibbs-Markov.
(Here, the return times $\tau$ and $\rho$ and the first return map $F$ are absent.)

Then there exists $p\in\N$ such that
$S_{v,w}(n)=O(n^{-(\beta-1)}\|v-v_0\|_{C^{\eta,p}}|w|_\infty)$ for all
observables $v,w$ with
$v\in C^{\eta,p}(X\times\T^d)$, $w\in L^\infty(X\times\T^d)$.
\end{thm}

\begin{rmk} \label{rmk-eigen}  
	We say that $v:X\times\T^d\to\R$ is a {\em trigonometric polynomial} if only finitely many of the Fourier coefficients $v_k:X\to \C$ in~\eqref{eq-vfourier} are nonzero.  

If at least one of the observables $v,w$ is a trigonometric polynomial, then all of our results simplify.
Instead of requiring nonexistence of approximate eigenfunctions, we require only the nonexistence of ordinary eigenfunctions (see Section~\ref{sec-ef}).  Moreover, we can take $p=0$.

In the simplified situation of trigonometric polynomials,
Theorem~\ref{thm-S2} recovers and improves upon~\cite{BHM05} where similar results are obtained only for $\beta>2$.
The improved convergence rate for observables supported in $Y$ in
Theorem~\ref{thm-S1} 
was also not obtained in~\cite{BHM05}.
\end{rmk}

All of our results about toral extensions $f_h$ are immediate consequences of
Theorems~\ref{thm-S1} and~\ref{thm-S2} combined with 
known results for $f$. In particular,  in the proofs of Theorem~\ref{thm-infinite} and Theorem \ref{thm-finite}
below we use~\eqref{eq-inf1}--\eqref{eq-fin}, while in the upper bounds result on decay of correlation,
namely Corollary~\ref{cor-upper} below, we use the result of Young~\cite{Young99}.

\begin{pfof}{Theorem~\ref{thm-infinite}}
Write 
$\int_{Y\times\T^d}v\,w\circ f_h^n\,dm= \int_Yv_0\, w_0\circ f^n \,d\mu+S_{v,w}(n)$.
For $\beta>\frac12$, by~\eqref{eq-inf1},
$\lim_{n\to\infty}
\tilde\ell(n) n^{1-\beta}\int_Yv_0\, w_0\circ f^n \,d\mu
=d_\beta\bar v_0\bar w_0 =d_\beta\bar v\bar w$.  
By Theorem~\ref{thm-S1}, $\tilde\ell(n)n^{1-\beta}S_{v,w}(n)=O(n^{1-2\beta+2\eps}\|v-v_0\|_{C^{\eta,p}}|w|_\infty)$.
Since $\beta>\frac12$ and $\eps$
is arbitrarily small, part (a) follows.

For $\beta\in(0,\frac12]$, or if $\bar v_0=0$ or $\bar w_0=0$, by~\eqref{eq-inf2},
$\int_Yv_0\, w_0\circ f^n \,d\mu=O(n^{-(\beta-\eps)}\|v_0\||w_0|_\infty)$.  
Hence part (b) follows from Theorem~\ref{thm-S1}.
\end{pfof}

\begin{pfof}{Theorem \ref{thm-finite}}
Write
$\int_{Y\times\T^d}v\,w\circ f_h^n\,dm-\bar v \bar w=g(n)+ S_{v,w}(n)$, where
$g(n)=
\int_Y v_0\;w_0\circ f^n\,d\mu
-\bar v_0\bar w_0$.
By~\eqref{eq-fin},
\[
g(n)= \sum_{j>n}\mu(\tau>j) \bar v_0\bar w_0+E_\beta(n)\|v_0\||w_0|_\infty
= \sum_{j>n}\mu(\tau>j) \bar v\bar w+E_\beta(n)\|v_0\||w_0|_\infty.
\]
The result follows from the estimates for $E_\beta(n)$ together with the estimates in
Theorem~\ref{thm-S1} for $S_{v,w}(n)$.
\end{pfof}

\begin{cor} \label{cor-upper}
Let $h:X\to\T^d$ be a $C^\eta$ cocycle, $\eta\in(0,1]$, and assume nonexistence of approximate eigenfunctions.
As in Theorem~\ref{thm-S2},
let $\varphi:Z\to\Z^+$ be a (general) return time such that 
$\mu_Z(\varphi>n)=O(n^{-\beta})$ where $\beta>1$, 
and $G=f^\varphi:Z\to Z$ is full branch Gibbs-Markov.

Then there exists $p\in\N$ such that
\[
	{\SMALL\int}_{X\times\T^d} v\,w\circ f_h^n\,dm-{\SMALL\int}_{X\times\T^d}v\,dm
{\SMALL\int}_{X\times\T^d}w\,dm
=O (n^{-(\beta-1)}
\|v\|_{C^{\eta,p}}|w|_\infty),
\]
for all $v\in C^{\eta,p}(X\times\T^d)$, $w\in L^\infty(X\times\T^d)$.
\end{cor}

\begin{proof}
Write
\[
{\SMALL\int_{X\times\T^d}}v\,w\circ f_h^n\,dm-{\SMALL\int_{X\times\T^d}} v\,dm {\SMALL\int_{X\times\T^d}} w\,dm=
{\SMALL\int_X} v_0\;w_0\circ f^n\,d\mu -{\SMALL\int_X} v_0\,d\mu{\SMALL\int_X} w_0\,d\mu
+ S_{v,w}(n). 
\]
By Young~\cite{Young99},
\[
|{\SMALL\int_X} v_0\;w_0\circ f^n\,d\mu -{\SMALL\int_X} v_0\,d\mu{\SMALL\int_X} w_0\,d\mu|
\le C n^{-(\beta-1)}\|v_0\|_{C^\eta} |w_0|_\infty ,
\]
for all $v_0$ H\"older and $w_0$ in $L^\infty$.
Hence the result follows from Theorem~\ref{thm-S2}.
\end{proof}

\section{Induced Gibbs-Markov maps}
\label{sec-GM}

Let $f:X\to X$ be a topologically mixing map satisfying 
conditions~(1)--(4) 
in assumption (iii) in Section~\ref{sec-intro}.
Let $G=f^\varphi:Z\to Z$ be the induced full branch Gibbs-Markov map as defined in assumption~(iii).
Standard references for background material on Gibbs-Markov maps are~\cite[Chapter~4]{Aaronson} and~\cite{AaronsonDenker01}.
In particular, a consequence of conditions~(1)--(3) is that there is a unique
ergodic $G$-invariant probability measure on $Z$ equivalent to $\mu_Z$ such that condition~(iii) still holds with this measure in place of $\mu_Z$.
Without loss we can suppose that $\mu_Z$ is this ergodic invariant probability measure.
Moreover $\mu_Z$ is mixing.
This leads to a unique (up to scaling) $f$-invariant measure $\mu$ on $X$ with $\mu|_Z$ equivalent
to $\mu_Z$, see for example~\cite[Theorem~1]{Young99}.  An
explicit definition of $\mu$ is given in Remark~\ref{rmk-mu}.
  The condition 
$\gcd\{\varphi(a):a\in\alpha\}=1$
implies that $f$ is topologically mixing, and in the finite measure case $\mu$ is mixing.

If $a_0,\dots,a_{n-1}\in\alpha$, we define the $n$-cylinder
$[a_0,\dots,a_{n-1}]=\bigcap_{j=0}^{n-1}G^{-j}a_j$.
Let $\theta\in(0,1)$ and define the symbolic
metric $d_\theta(z,z')=\theta^{s(z,z')}$ where the {\em separation time}
$s(z,z')$ is the greatest integer $n\ge0$ such that $z$ and $z'$ lie in the same $n$-cylinder.  In the remainder of this section, we fix $\theta\in[\lambda^{-\eta},1)$.
For convenience we rescale the metric $d$ on $X$ so that $\diam Z\le1$.

\begin{prop} \label{prop-d}
$d(z,z')^\eta\le d_\theta(z,z')$ for all $z,z'\in Z$.
\end{prop}

\begin{proof}
Let $n=s(z,z')$.  By condition~(2),
\[
1\ge \diam Z\ge d(G^nz,G^nz')\ge \lambda^nd(z,z')\ge (\theta^{1/\eta})^{-n}d(z,z').
\]
Hence $d(z,z')^\eta\le  \theta^n=d_\theta(z,z')$.
\end{proof}

An observable $v:Z\to\R$ is {\em Lipschitz} if $\|v\|_\theta=|v|_\infty+|v|_\theta<\infty$ where
$|v|_\theta=\sup_{z\neq z'}|v(z)-v(z')|/d_\theta(z,z')$.
The set $F_\theta(Z)$ of Lipschitz observables is a Banach space.
More generally, we say that $v:Z\to\R$ is {\em locally Lipschitz}, and write
$v\in F_\theta^{\rm loc}(Z)$, if $v|_a\in F_\theta(a)$ for each $a\in\alpha$.  Accordingly, we define
$D_\theta v(a)=\sup_{z,z'\in a:\,z\neq z'}|v(z)-v(z')|/d_\theta(z,z')$.

We say that an observable $v=(v_1,\dots,v_d):Z\to\R^d$ lies in $F_\theta(Z,\R^d)$ if
$v_1,\dots,v_d\in F_\theta(Z)$, and we define
$|v|_\theta=\max_{j=1,\dots,d}|v_j|_\theta$
and $\|v\|_\theta= \max_{j=1,\dots,d}\|v_j\|_\theta$.
Similarly, we define $F_\theta^{\rm loc}(Z,\R^d)$ and
$F_\theta^{\rm loc}(Z,\T^d)$.

%

\begin{prop}  \label{prop-DH}
Let $h:X\to\T^d$ be a $C^\eta$ cocycle.  Define the 
induced cocycle 
$H(z)=\sum_{\ell=0}^{\varphi(z)-1}h(f^jz)$.  
Then
$H\in F_\theta^{\rm loc}(Z,\T^d)$, and
there is a constant $C_2\ge1$ such that
\begin{align*}
 D_\theta H(a)\le C_2|h|_{C^\eta} \varphi(a),
\end{align*}
for all $a\in\alpha$.
\end{prop}

\begin{proof}
Let $z,z'\in a$.  Then $\varphi(z)=\varphi(z')=\varphi(a)$.  
Let $C_1'=C_1^\eta$.  By condition~(4) and Proposition~\ref{prop-d},
\begin{align*}
& |H(z)-H(z')|
 \le \sum_{\ell=0}^{\varphi(a)-1}|h(f^\ell z)-h(f^\ell z')|
\le |h|_{C^\eta}\sum_{\ell=0}^{\varphi(a)-1}d(f^\ell z,f^\ell z')^\eta
\\ & \qquad  \le C_1' |h|_{C^\eta} \varphi(a)d(Gz,Gz')^\eta
\le C_1' |h|_{C^\eta} \varphi(a)d_\theta(Gz,Gz')
= C_1' \theta^{-1}|h|_{C^\eta} \varphi(a)d_\theta(z,z'),
\end{align*}
yielding the required estimate for $D_\theta H(a)$.
\end{proof}

The transfer operator $R:L^1(Z)\to L^1(Z)$ corresponding to the induced map
$G:Z\to Z$ is given by
$\int_Z Rv\,w\,d\mu_Z=\int_Z v\,w\circ G\,d\mu_Z$ for all $v\in L^1(Z)$,
$w\in L^\infty(Z)$.
Since we are now taking $\mu_Z$ to be invariant, this is the normalized transfer operator satisfying $R1=1$.
It can be easily seen that $(Rv)(z)=\sum_{a\in\alpha}e^{g(z_a)}v(z_a)$
where $z_a$ denotes the unique preimage of $z$ in $a$ under $G$ and $g$ is the potential defined in condition~(3) in the definition of Gibbs-Markov map (beginning of Section~\ref{sec-toral}). 
Similarly, $(R^nv)(z)=\sum_{a\in\alpha_n}e^{g_n(z_a)}v(z_a)$
where $z_a$ denotes the unique preimage of $z$ in $a$ under $G^n$ and $g_n(z)=\sum_{j=0}^{n-1}g(G^jz)$.
Moreover, there exists a constant $C_3$ such that
\begin{align} \label{eq-GM}
	e^{g_n(z)}\le C_3\mu_Z(a), \quad\text{and}\quad |e^{g_n(z)}-e^{g_n(z')}|\le C_3\mu_Z(a)d_\theta(G^nz,G^nz'),
\end{align}
for all $z,z'\in a$, $a\in\alpha_n$, $n\ge1$.

\begin{prop}  \label{prop-basic}
There exists $\tau\in(0,1)$ such that
$\|R^n v-\int_Zv\,d\mu_Z\|_\theta \le C\tau^n\|v\|_\theta$,
for all $n\ge1$ and $v\in F_\theta(Z)$.
\end{prop}

\begin{proof}  
This follows from the fact that the transfer operator $R$ has a spectral 
gap~\cite[Section~4.7]{Aaronson}.
\end{proof}

\section{Eigenfunctions and approximate eigenfunctions}
\label{sec-eigen}

In this section, we recall the notion of approximate eigenfunction, and show that typically there are none.  That is,  condition~(v) in the introduction holds typically.

In Subsection~\ref{sec-ef}, we consider ordinary eigenfunctions as mentioned in Remark~\ref{rmk-eigen}. (Non-existence of eigenfunctions is a sufficient condition
for a technical result on renewal operators, namely
Proposition~\ref{prop-fourier}, required in the proof of our main results.)
Approximate eigenfunctions are then considered in Subsection~\ref{sec-approx}.

Throughout this section, we work with toral extensions of a map $f:X\to X$ with
full branch Gibbs-Markov induced map $G=f^\varphi:Z\to Z$ corresponding to a general return time $\varphi:Z\to\Z^+$.  Given a measurable cocycle $h:X\to\T^d$, 
we define the induced cocycle $H:Z\to\T^d$ given by
$H(z)=\sum_{\ell=0}^{\varphi(z)-1}h(f^\ell z)$.

\subsection{Eigenfunctions}
\label{sec-ef}

In this subsection, we define eigenfunctions and recall some of their basic properties.
Let $S^1$ denote the unit circle in $\C$.

\begin{defn} \label{def-ef} A measurable function $v:Z\to S^1$ is an {\em eigenfunction}
if there exist {\em frequencies} $k\in\Z^d\setminus\{0\}$ and $\omega\in[0,2\pi)$ such that
$v\circ G=e^{ik\cdot H}e^{i\omega\varphi}v$.
\end{defn}

By Remark~\ref{rmk-eigen}, nonexistence of eigenfunctions is a sufficient condition for our main results in the case of trigonometric polynomials.  The next result shows that nonexistence of eigenfunctions is typical.

\begin{prop} \label{prop-typical}
Suppose that $h$ is $C^\eta$ for some $\eta>0$ and that $z_1$ and $z_2$ are fixed points for $G:Z\to Z$.
If there exists an eigenfunction, then there exist $k_1,k_2\in\Z^d\setminus\{0\}$ such that $k_1\cdot H(z_1)= k_2\cdot H(z_2)\bmod 2\pi$.  
\end{prop}

\begin{proof}
Suppose that $v$ is an eigenfunction with frequencies $k\in\Z^d\setminus\{0\}$ and $\omega\in[0,2\pi)$.  
Since $G$ is Gibbs-Markov, it follows by Liv\v sic regularity that $v$ is continuous.
Since $z_j$ are fixed points, we obtain
$e^{ik\cdot H(z_j)}e^{i\omega \varphi(z_j)}=1$.
Hence
$e^{i\varphi(z_2)k\cdot H(z_1)}e^{i\omega \varphi(z_1)\varphi(z_2)}=1$ and
$e^{i\varphi(z_1)k\cdot H(z_2)}e^{i\omega \varphi(z_1)\varphi(z_2)}=1$.  It follows that
$e^{i\varphi(z_2)k\cdot H(z_1)}=
e^{i\varphi(z_1)k\cdot H(z_2)}$.  The result follows with 
$k_1=\varphi(z_2)k$ and $k_2=\varphi(z_1)k$.~
\end{proof}

It follows that nonexistence of eigenfunctions holds generically (for a residual set of $C^\eta$ cocycles $h:X\to\T^d$ for any fixed $\eta>0$).
An open and dense criterion is given in Appendix~\ref{app-good}.

\subsection{Approximate eigenfunctions}
\label{sec-approx}

For $k\in\Z^d$, $\omega\in[0,2\pi]$, define $M_{k,\omega}:L^\infty(Z)\to L^\infty(Z)$,
\[
M_{k,\omega}v=e^{-ik\cdot H}e^{-i\omega\varphi}v\circ G.
\]
Note that $v$ is an eigenfunction with frequencies $k$, $\omega$ if and only if $M_{k,\omega}v=v$.

\begin{defn}  
\label{def-approx}
There are {\em approximate eigenfunctions on a subset $Z_\infty\subset Z$}
if for any $\xi_0>0$, there exist constants $\xi,\,\zeta>\xi_0$ and $C\ge1$, and
sequences \[
u_j\in F_\theta(Z),\; k_j\in\Z^d\setminus\{0\},\; \omega_j\in[0,2\pi),\; \chi_j\in[0,2\pi),\;
n_j=[\zeta\ln|k_j|]\in\N,\;
 j\ge1,
\]
with $\lim_{j\to\infty}|k_j|=\infty$,
$|u_j|\equiv1$ and $|u_j|_\theta\le C|k_j|$, such that 
\[
|(M_{k_j,\omega_j}^{n_j}u_j)(z)-e^{i\chi_j}u_j(z)|\le C|k_j|^{-\xi},
\]
for all $z\in Z_\infty$ and all $j\ge1$.
\end{defn}

\begin{defn} A subset $Z_\infty\subset Z$ is called a {\em finite subsystem} of $Z$
if $Z_\infty=\bigcap_{n\ge1}G^{-n}Z_0$ where $Z_0$ is the union of finitely many elements
from the partition $\alpha$.
\end{defn}

\begin{defn}  We say that {\em there exist approximate eigenfunctions} if for every finite subsystem $Z_\infty\subset Z$ there exist approximate eigenfunctions on $Z_\infty$.
\end{defn}

\begin{prop} \label{prop-dio}
Let $z_1,z_2,z_3$ be three fixed points for $G:Z\to Z$ such that $\varphi(z_1)\neq\varphi(z_2)$.  
Let $Z_\infty$ be the finite subsystem corresponding to the union of the partition elements containing $z_1,z_2,z_3$.
For almost all $H(z_1),H(z_2),H(z_3)\in\T^d$, there are no approximate eigenfunctions on $Z_\infty$.
\end{prop}

\begin{proof}
Suppose that there exist approximate eigenfunctions on $Z_\infty$.  Then there exists sequences as in Definition~\ref{def-approx} such that
\[
|e^{-in_jk_j\cdot H(z_a)}e^{-in_j\omega_j\varphi(z_a)}-e^{i\chi_j}|=O(|k_j|^{-(d+2)}),
\]
for $a=1,2,3$, $j\ge1$.
Eliminating $\chi_j$, we obtain
\[
\dist(n_jk_j\cdot (H(z_a)-H(z_3))+n_j\omega_j(\varphi(z_a)-\varphi(z_3)),2\pi\Z)=O(|k_j|^{-(d+2)}),
\]
for $a=1,2$, $j\ge1$.
Define $\tilde k_j=n_j k_j\in\Z^d$ and
$\Omega= (\varphi(z_1)-\varphi(z_3))(H(z_2)-H(z_3))-(\varphi(z_2)-\varphi(z_3))(H(z_1)-H(z_3))$.
Eliminating $\omega_j$, we obtain
\[
\dist(\tilde k_j\Omega,2\pi\Z)=O(|k_j|^{-(d+2)}).
\]
For almost every value of $\Omega$, this Diophantine condition holds for at most finitely many values of $\tilde k_j\in\Z^d$, violating the requirement that $|k_j|\to\infty$.  Hence approximate eigenfunctions do not exist on $Z_\infty$.
\end{proof}

Field~{\em et al.}~\cite{FMT05,FMT07} introduced the notion of {\em good asymptotics}.  We recall this notion in Appendix~\ref{app-good} and show that it gives an open and dense criterion for nonexistence of approximate eigenfunctions for (piecewise) smooth toral extensions, including those in Examples~\ref{eg-1} and~\ref{eg-2}.

\section{Fourier analysis and H\"older norms}
\label{sec-F}

In this section, we recall some standard results about smoothness of Fourier 
series~\cite{Katzn}.
Let $A_n$ be a sequence of bounded linear
operators on some Banach space $\cX$ and set $A(\omega)=\sum_{n=1}^\infty
A_n e^{in\omega}$, $\omega\in[0,2\pi]$.
If $A\in L^1$ then we
define the Fourier coefficients $\hat A_n=(1/2\pi)\int_0^{2\pi}e^{-in\omega} A(\omega)\,d\omega$.

When speaking of regularity of $A$, we regard $A$ as a $2\pi$-periodic function on $\R$.
Let $|A|_{C^0}=\sup_\omega\|A(\omega)\|$.
For $m\in\N$, define 
$\|A\|_{C^m}=\max_{j=0,\dots,m}|A^{(j)}|_{C^0}$.
For $q=m+\alpha$, $m\in\N$, $\alpha\in[0,1)$,
define $\|A\|_{C^q}=\|A\|_{C^m}+|A^{(m)}|_\alpha$ where
$|A|_\alpha=\sup_{\omega_1\neq\omega_2}|A(\omega_1)-A(\omega_2)|/|\omega_1-\omega_2|^\alpha$.

\begin{prop}  \label{prop-holder}
Suppose that $\sum_{j>n}\|A_j\|\le Cn^{-q}$ for constants $C\ge1$, $q>0$,
where $q$ is not an integer.
Then there is a universal constant $D_q$ depending only on $q$ such that
$A:[0,2\pi]\to L(\cX,\cX)$ is $C^q$ and $\|A\|_{C^q}\le CD_q$.
\end{prop}

\begin{proof}  The details are written out for example
in~\cite[Lemma~2.4]{BHM05}.
\end{proof}

\begin{prop}  \label{prop-coeff}
Suppose that $A:[0,2\pi]\to L(\cX,\cX)$ is $C^q$, $q>0$.
Then there is a universal constant $D_q$ depending only on $q$ such that
$\|\hat A_n\|\le D_q \|A\|_{C^q} n^{-q}$.
\end{prop}

\begin{proof}  The details are written out for example
in~\cite[Lemma~2.5]{BHM05}.
\end{proof}

\begin{rmk}\label{rmk-fourier}
If $q>1$ in Propositions~\ref{prop-holder} or~\ref{prop-coeff}, then $A_n=\hat A_n$ and the Fourier series is uniformly absolutely convergent.
\end{rmk}

Next, we consider H\"older norms  of families of operator functions
$A,\,B:[0,2\pi]\to L(\cX,\cX)$ where $A(\omega)$ is invertible for all $\omega\in[0,2\pi]$ and $B(\omega)=A(\omega)^{-1}$.

\begin{lemma}  \label{lem-inverse}
For each $m\in\N$, there is a universal constant $c_m>0$ such 
that for all $q=m+\alpha$, $\alpha\in[0,1)$,
\[
\|B\|_{C^q}\le c_m (1+\|B\|_{C^0})^{2q+2}(1+\|A\|_{C^q})^{2q+1}.
\]
\end{lemma}

\begin{proof}  
First we consider the case $q=m\in\N$.  The case $m=0$ is trivial.
For $m\ge1$, note that
$D^mB$ is a linear combination of terms of the form
$(D^{n_1}B)(D^{n_2}A)(D^{n_3}B)$ with $n_1+n_2+n_3=m$ and $n_2\ge1$.
Inductively,
\begin{align*}
 |D^mB|_{C^0}  & \le c_m' \!\!\!\!\sum_{\stackrel{ n_1+n_2+n_3=m }{ n_2\ge1 }} \!\!\!\!
|D^{n_1}B|_{C^0} |D^{n_2}A|_{C^0} |D^{n_3}B|_{C^0} 
 \le c_m'\|A\|_{C^m} \!\!\!\! \sum_{n_1+n_3\le m-1} \!\!\!\!
\|B\|_{C^{n_1}} \|B\|_{C^{n_3}} 
\\ & \le c_m''\|A\|_{C^m} \!\!\!\!\sum_{n_1+n_3\le m-1} \!\!\!\!
(1+\|B\|_{C^0})^{2n_1+2n_3+4}(1+\|A\|_{C^m})^{2n_1+2n_3+2}
\\ & \le c_m''' 
(1+\|B\|_{C^0})^{2m+2}(1+\|A\|_{C^m})^{2m+1},
\end{align*}
establishing the required result when $q=m$ is an integer.

When $q=m+\alpha$, we have in addition that
\begin{align*}
 |D^mB|_\alpha  & \le c_m' \!\!\!\!\sum_{\stackrel{ n_1+n_2+n_3=m }{ n_2\ge1 }} \!\!\!\!
(2|D^{n_1}B|_\alpha |D^{n_2}A|_{C^0} |D^{n_3}B|_{C^0} 
 + |D^{n_1}B|_{C^0} |D^{n_2}A|_\alpha |D^{n_3}B|_{C^0} )
\\
& \le 3c_m'\|A\|_{C^q} \!\!\!\!\sum_{n_1+n_3\le m-1} \!\!\!\!\|B\|_{C^{n_1+\alpha}}\|B\|_{C^{n_3}}
\\ & \le c_m''\|A\|_{C^q} \!\!\!\!\sum_{n_1+n_3\le m-1} \!\!\!\!
(1+\|B\|_{C^0})^{2n_1+2\alpha+2n_3+4}(1+\|A\|_{C^q})^{2n_1+2\alpha+2n_3+2}
\\ & \le c_m''' 
(1+\|B\|_{C^0})^{2q+2}(1+\|A\|_{C^q})^{2q+1},
\end{align*}
completing the proof.
\end{proof}

\section{Estimates for induced twisted transfer operators}
\label{sec-induced}

Throughout this section, we assume condition~(iii) on the induced map
$G=f^\varphi:Z\to Z$.
Recall from Section~\ref{sec-GM} that
$R$ is the transfer operator
corresponding to $G$,
and that $H:Z\to\T^d$ is the induced cocycle $H(z)=\sum_{\ell=0}^{\varphi(z)-1}h(f^\ell y)$.  

For $k\in\Z^d$,
define the twisted transfer operators
$R_k:L^1(Z)\to L^1(Z)$,
$R_kv=R(e^{ik\cdot H}v)$.
We can write $R_k=\sum_{n=1}^\infty R_{k,n}$ where $R_{k,n}:L^1(Z)\to L^1(Z)$
is given by
\[
R_{k,n}v=R_k(1_{\{\varphi=n\}}v).
\]
Define $R_k(\omega):L^1(Z)\to L^1(Z)$ for $\omega\in[0,2\pi]$ 
by setting
\[
R_k(\omega)v=\sum_{n=1}^\infty R_{k,n}e^{in\omega}v=
R(e^{ik\cdot H}e^{i\omega\varphi}v).
\]
Note that
$R_k(\omega)^nv=R^n(e^{ik\cdot H_n}e^{i\omega\varphi_n}v)$
where $H_n=\sum_{j=0}^{n-1}H\circ G^j$,
$\varphi_n=\sum_{j=0}^{n-1}\varphi\circ G^j$.

In Subsection~\ref{sec-basic}, we derive some basic estimates for the operators $R_{k,n}$ and $R_k(\omega)$.
In Subsection~\ref{sec-Dolg}, we obtain a Dolgopyat-type estimate.

\subsection{Some basic estimates}
\label{sec-basic}

\begin{prop}  \label{prop-infty}
For all $k\in\Z^d$, $\omega\in[0,2\pi]$,  $n\ge1$,
(a) $|R_{k,n}|_\infty \le C_3 \mu_Z(\varphi=n)$ and 
(b) $|R_k(\omega)|_\infty\le1$.
\end{prop}

\begin{proof}
Let $z\in Z$.  For each $a\in\alpha$, let $z_a$ be the unique
preimage $z_a\in a\cap G^{-1}(z)$.  Then 
\[
	(R_{k,n}v)(z)=\sum_{a\in\alpha:\varphi(a)=n}e^{g(z_a)}e^{ik\cdot H(z_a)}v(z_a).
\]
By~\eqref{eq-GM},
\[
|R_{k,n}v|_\infty\le C_3 \sum_{a\in\alpha:\varphi(a)=n}\mu_Z(a)|v|_\infty
= C_3\mu_Z(\varphi=n)|v|_\infty,
\]
proving part~(a).

Since $|R|_\infty=1$ and $R_k(\omega)v=R(e^{ik\cdot H}e^{i\omega\varphi}v)$,
part (b) is immediate.
\end{proof}

\begin{lemma} \label{lem-Rk}
Let $\eps>0$ and fix $\theta\in[\lambda^{-\eta\eps},1)$.
There exists a constant $C\ge1$ such that for 
every $v\in F_\theta(Z)$, $k\in\Z^d\setminus\{0\}$, 
$\omega\in[0,2\pi]$,
and for every $n$-cylinder $a\in\alpha_n$, $n\ge1$,
\[
|R_k(\omega)^n(1_av)|_\theta \le C\mu_Z(a)\Bigl\{
|k|^\eps \sum_{j=0}^{n-1}
\theta^{n-j}\varphi(G^ja)^\eps|v|_\infty +\theta^n|v|_\theta\Bigr\}.
\]
\end{lemma}

\begin{proof}
Let $z\in Z$, and  let $z_a$ be the unique
preimage $z_a\in a\cap G^{-n}(z)$.  Noting that $\varphi_n$ is constant on $a$,
\[
	(R_k(\omega)^n(1_av))(z)=e^{g_n(z_a)}e^{ik\cdot H_n(z_a)}e^{i\omega\varphi_n(a)}v(z_a),
\]
and
\[
(R_k(\omega)^n(1_av ))(z)-(R_k(\omega)^n(1_av ))(z')=I_1+I_2+I_3,
\]
where 
\begin{align*}
	I_1 & = (e^{g_n}(z_a)-e^{g_n}(z_a'))e^{ik\cdot H_n(z_a)}e^{i\omega\varphi_n(a)}v(z_a), \\
	I_2 & = e^{g_n}(z_a')(e^{ik\cdot H_n(z_a)}-e^{ik\cdot H_n(z_a')})e^{i\omega\varphi_n(a)}v(z_a), \\
	I_3 & = e^{g_n}(z_a')e^{ik\cdot H_n(z_a')}e^{i\omega\varphi_n(a)}(v(z_a)-v(z_a')).
\end{align*}
By~\eqref{eq-GM},
\begin{align*}
|I_1| & \le C_3 \mu_Z(a)|v|_\infty d_\theta(z,z'), \quad
|I_3|  \le C_3 \mu_Z(a)|v|_\theta d_\theta(z_a,z_a') = C_3\theta^n\mu_Z(a)|v|_\theta d_\theta(z,z').
\end{align*}
Using the inequality $|e^{ix}-1|\le 2|x|^\eps$ for all $x\in\R$, $\eps\in(0,1]$, 
\[
|I_2|\le 2C_3 \mu_Z(a)|k|^\eps|H_n(z_a)-H_n(z_a')|^\eps |v|_\infty.
\]
Let $\gamma\in[\lambda^{-\eta},1)$.  
By definition of $z_a,z_a'$,
\[
d_\gamma(G^jz_a,G^jz_a')=\gamma^{-j}d_\gamma(z_a,z_a')=\gamma^{n-j}d_\gamma(z,z'),
\]
for $j=0,\dots,n-1$, and so  
by Proposition~\ref{prop-DH} (with $\theta=\gamma$),
\begin{align*}
|H(G^jz_a)-H(G^jz_a')| & \le D_\gamma H(G^ja)d_\gamma(G^jz_a,G^jz_a')\le C_2|h|_{C^\eta} \varphi(G^ja)\gamma^{n-j}d_\gamma(z,z').
\end{align*}
Hence
\begin{align*}
|H_n(z_a)-H_n(z_a')| =\Bigl|\sum_{j=0}^{n-1}(H(G^jz_a)-H(G^jz_a'))\Bigr|
\le C_2|h|_{C^\eta} \sum_{j=0}^{n-1}\gamma^{n-j}\varphi(G^ja)\gamma^{s(z,z')}.
\end{align*}
It follows that
\begin{align*}
|I_2| & \le 2C_2C_3 |k|^\eps |v|_\infty\mu_Z(a)|h|_{C^\eta}^\eps
\Bigl|\sum_{j=0}^{n-1}\gamma^{n-j}\varphi(G^ja)\Bigr|^\eps \gamma^{\eps s(z,z')}
\\ & \le 2C_2C_3 |k|^\eps |v|_\infty\mu_Z(a)|h|_{C^\eta}^\eps
\sum_{j=0}^{n-1}\gamma^{\eps(n-j)}\varphi(G^ja)^\eps \gamma^{\eps s(z,z')}.
\end{align*}
Choosing $\gamma=\theta^{1/\eps}$,
\[
|I_2|\le 2C_2C_3 |k|^\eps |v|_\infty\mu_Z(a)|h|_{C^\eta}^\eps
\sum_{j=0}^{n-1}\theta^{n-j}\varphi(G^ja)^\eps d_\theta(z,z').
\]
Combining the estimates for $I_1,I_2,I_3$ yields the required result.
\end{proof}

\begin{cor} \label{cor-Rk}
Choose $\eps$ such that $\varphi^\eps\in L^1(Z)$
and let $\theta\in[\lambda^{-\eta\eps},1)$.
There exists a constant $C_4\ge1$ (depending on $h$, $\varphi$, $\eps$) such that for 
every $\theta\in(0,1)$, $v\in F_\theta(Z)$, $k\in\Z^d\setminus\{0\}$, 
$\omega\in[0,2\pi]$,
$n\ge1$,
\begin{itemize}
\item[(a)]
$\|R_{k,n}\|_\theta  \le C_4\mu_Z(\varphi=n)|k|^\eps n^\eps$.
\item [(b)] 
$|R_k(\omega)^nv|_\theta \le C_4\{|k|^\eps|v|_\infty+\theta^n|v|_\theta\}$.
\end{itemize}
\end{cor}

\begin{proof}  Taking $\omega=0$, $n=1$, $a\in\alpha$ in 
Lemma~\ref{lem-Rk}, we obtain
\[
|R_k(1_av)|_\theta 
\le C\mu_Z(a)\{|k|^\eps \varphi(a)^\eps)|v|_\infty +|v|_\theta\}
\le C\mu_Z(a)|k|^\eps \varphi(a)^\eps\|v\|_\theta.
\]
Summing over those $a$ with $\varphi(a)=n$,
we obtain that $|R_{k,n}v|_\theta  \ll \mu_Z(\varphi=n)|k|^\eps n^\eps\|v\|_\theta$.
 This combined with
Proposition~\ref{prop-infty}(a) yields part (a).

To prove part (b), we write $R_k(\omega)^nv=\sum_{a\in\alpha_n}R_k(\omega)^n(1_av)$ and sum the estimates from Lemma~\ref{lem-Rk}.
Note that
\begin{align*}
\sum_{a\in\alpha_n} \mu_Z(a)\sum_{j=0}^{n-1}\theta^{n-j}\varphi(G^ja)^\eps
& = \sum_{j=0}^{n-1}\theta^{n-j}\sum_{b\in\alpha_{n-j}}\varphi(b)^\eps \sum_{a\in\alpha_n:G^ja=b} \mu_Z(a) \\
& = \sum_{j=0}^{n-1}\theta^{n-j}\sum_{b\in\alpha_{n-j}}\varphi(b)^\eps \mu_Z(b) 
 \le \theta(1-\theta)^{-1} \sum_{b\in\alpha}\mu_Z(b)\varphi(b)^\eps.
\end{align*}
Hence $|R_k(\omega)^nv|_\theta \le C\{ \theta(1-\theta)^{-1} 
|k|^\eps \sum_{a\in\alpha}\mu_Z(a)\varphi(a)^\eps
|v|_\infty +\theta^n|v|_\theta\}$.
\end{proof}

\begin{cor} \label{cor-RkHolder}
Choose $\eps\in(0,\beta)$ so that $\beta-\eps$ is not an integer and such that $\varphi^\eps\in L^1(Z)$.
Let $\theta\in[\lambda^{-\eta\eps},1)$.

For each $k\in\Z^d\setminus\{0\}$, the map $R_k:[0,2\pi]\to L(F_\theta(Z),F_\theta(Z))$,
$\omega\mapsto R_k(\omega)$, is $C^{\beta-\eps}$.
Moreover, there is a constant $C\ge1$ independent of $k$ such that
$\|R_k\|_{C^{\beta-\eps}}\le C|k|^\eps$.
\end{cor}

\begin{proof}
	Recall from Remark~\ref{rmk:beta} that $\mu_Z(\varphi>n)=O(n^{-(\beta-\eps)})$.
By Corollary~\ref{cor-Rk}(a), we have that
$\sum_{j>n}\|R_{k,j}\|_\theta \ll |k|^\eps n^{-(\beta-2\eps)}$.
Now apply Proposition~\ref{prop-holder}.
\end{proof}

\subsection{A Dolgopyat-type estimate}
\label{sec-Dolg}

The argument in this subsection is
a direct adaptation of an  argument in~\cite{M07} and is included mainly for completeness.  Propositions~\ref{prop-approx1} and~\ref{prop-approx2} below correspond 
to~\cite[Lemmas~3.12 and~3.13]{M07} respectively, and the Dolgopyat-type estimate, Lemma~\ref{lem-approx}, follows immediately.

Throughout, we fix $\eps\in(0,1]$ such that $\varphi^\eps\in L^1(Z)$, and  
$\theta\in[\lambda^{-\eta\eps},1)$.

\begin{rmk}  \label{rmk-basic}
As in~\cite[Section~6]{Dolgopyat98b}, we
define $\|v\|_k=\max\{|v|_\infty,|v|_\theta/(2C_4|k|^\eps)\}$.
Then it follows from Proposition~\ref{prop-infty}(a) and Corollary~\ref{cor-Rk}(a) that
$\|R_k(\omega)^n\|_k\le C_4+\frac12$ for all $n\ge 1$.  Moreover,
$\|R_k(\omega)^n\|_k\le 1$ for all $n\ge n_0$
(where $n_0=[\ln(2C_4)/(-\ln\theta)]+1$).
\end{rmk}

Since we are estimating operator norms with respect to $\|\;\|_k$,
we consider the unit ball $F_\theta(Z)_k=\{v\in F_\theta:\|v\|_k\le1\}$.
It follows from Remark~\ref{rmk-basic} that
$|R_k(\omega)^nv|_\infty\le 1$
and $|R_k(\omega)^nv|_\theta\le 2C_4|k|^\eps$ for all $v\in F_\theta(Z)_k$
and $n\ge n_0$.

Throughout, $Z_0$ denotes a fixed subset of $Z$ consisting of a finite union of
partition elements of $Z$, and $Z_\infty=\bigcap_{j\ge0}G^{-j} Z_0$.
Note that the potential $g$ is uniformly bounded on $Z_\infty$ and moreover
$g_n(z)\le n|1_{Z_\infty}g|_\infty$ for all $z\in Z_\infty$ and $n\ge1$.

\begin{prop}  \label{prop-approx1}
Let $\xi_2,\,\zeta_0>0$.  Then there exist $\xi_1>0$ and $\zeta>\zeta_0$,
such that the following is true for each fixed $|k|\ge2$, $\omega\in[0,2\pi]$,
setting $n(k)=[\zeta\ln |k|]$:

Suppose that there exists $v_0\in F_\theta(Z)_k$ such that
for all $x\in Z_\infty$ and all $j=0,1,2$,
\[
|(R_k(\omega)^{jn(k)}v_0)(x)|\ge 1-1/|k|^{\xi_1}.
\]
Then there exists $w\in F_\theta(Z)$ with
$|w|\equiv1$, $|w|_\theta\le 16C_4|k|$, and $\chi\in[0,2\pi)$ such that for all $z\in Z_\infty$,
\[
|(M_{k,\omega}^{n(k)}w)(z)-e^{i\chi}w(z)| \le 8/|k|^{\xi_2}.
\]
\end{prop}

\begin{proof}
We write $n=n(k)$ and $\tilde C_4=16C_4$.  Set
\[
\zeta=(\xi_2+2+\ln \tilde C_4/\ln 2)/(-\ln\theta), \qquad 
\xi_1=\max\{1,2\xi_2+\zeta|1_{Z_\infty}g|_\infty\}.
\]
If necessary, increase $\zeta$ so that $\zeta>\zeta_0$.
Following~\cite[Section~8]{Dolgopyat98b} and~\cite[Section~3]{M07}, we write
 $v_j=R_k(\omega)^{jn}v_0$ and $v_j=s_jw_j$, where $|w_j(x)|\equiv1$
and $1-1/|k|^{\xi_1}\le s_j(x)\le 1$ for $x\in Z_\infty$.
Note that $|v_j|_\theta\le 2C_4|k|^\eps$ so that $|w_j|_\theta\le \tilde C_4|k|^\eps\le \tilde C_4|k|$.
Rearrange $v_1=R_k(\omega)^nv_0$ to obtain
$w_1^{-1}R_k(\omega)^n(s_0w_0)=s_1\ge 1-1/|k|^{\xi_1}$.
It then follows from the definition of $R_k(\omega)$ that
$e^{g_n(z)}[1-\Re(e^{ik\cdot H_n(z)}e^{i\omega\varphi_n(z)}w_0(z)w_1^{-1}(G^nz))]\le 1/|k|^{\xi_1}$
for all $z\in Z$ with $G^nz\in Z_\infty$.  Hence
$|e^{ik\cdot H_n(z)}e^{i\omega\varphi_n(z)}w_0(z)-w_1(G^nz)|\le 2(e^{-g_n(z)}/|k|^{\xi_1})^{1/2}$.
Similarly, with $w_0$ and $w_1$ replaced by $w_1$ and $w_2$.
Restricting to $z\in Z_\infty$, we have
$e^{-g_n(z)}/|k|^{\xi_1}\le 1/|k|^{2\xi_2}$ and hence
\begin{align} \nonumber
|e^{ik\cdot H_n(z)}e^{i\omega\varphi_n(z)}w_0(z)-w_1(G^nz)| &\le 2/|k|^{\xi_2}, 
\\ 
|e^{ik\cdot H_n(z)}e^{i\omega\varphi_n(z)}w_1(z)-w_2(G^nz)| & \le 2/|k|^{\xi_2},
\label{eq-w}
\end{align}
for all $z\in Z_\infty$.
Fix $q\in Z_\infty$ and choose $\chi_0$, $\chi_1\in\R$ such that $w_j(q)=e^{i\chi_j}$ for $j=0,1$ and such that $\chi=\chi_0-\chi_1\in[0,2\pi)$.
To each $z$, we associate $z^*=q_0\cdots q_{n-1}z_nz_{n+1}\cdots\in Z_\infty$.
Then $z^*$ is within distance $\theta^n$ of $q$ and $G^nz^*=G^nz$.
We obtain
\begin{align*}
|e^{ik\cdot H_n(z^*)}e^{i\omega\varphi_n(z^*)}e^{i\chi_0}-w_1(G^nz)| &\le 2/|k|^{\xi_2}+\tilde C_4|k|\theta^n \le 3/|k|^{\xi_2}\\
|e^{ik\cdot H_n(z^*)}e^{i\omega\varphi_n(z^*)}e^{i\chi_1}-w_2(G^nz)| &\le 2/|k|^{\xi_2}+\tilde C_4|k|\theta^n\le3/|k|^{\xi_2}
\end{align*}
(by the choice of $\zeta$), and so
$|e^{-i\chi}w_1(G^nz)-w_2(G^nz)| \le 6/|k|^{\xi_2}$.
Substituting into~\eqref{eq-w} yields the required approximate eigenfunction
$w=w_1$.
\end{proof}

\begin{prop}  \label{prop-approx2}
For any $\xi_1,\zeta>0$, there exists $\xi>0$ and $C\ge1$ with the
following property.

Let $|k|\ge1$ and suppose that for any
$v\in F_\theta(Z)_k$
there exists $x_0\in Z_\infty$ and $j\le [\zeta\ln |k|]$
such that $|(R_k(\omega)^j v)(x_0)|\le 1-1/|k|^{\xi_1}$.
Then $\|(I-R_k(\omega))^{-1}\|_k \le C|k|^\xi$.
\end{prop}

\begin{proof}
Following~\cite[Section~7]{Dolgopyat98b}, we use the pointwise estimate on
iterates
of $R_k(\omega)$ to obtain estimates on the $L^1$, $L^\infty$ and $\|\;\|_k$ norms.

Write $\hat u=R_k(\omega)^jv$ and $u=R_k(\omega)^{\ell(k)}v$
where $\ell(k)=[\zeta\ln |k|]$.  Note that $|\hat u|_\infty\le1$
and $|\hat u|_\theta\le 2C_4|k|^\eps\le 2C_4|k|$.  Hence, $|\hat u(x)|\le 1-1/(2|k|^{\xi_1})$ for all $x$
within distance $1/(4C_4|k|^{\xi_1+1})$ of $x_0$.  Call this subset $U$.
If $\mathcal{C}_m$ is an $m$-cylinder, then $\diam\mathcal{C}_m=\theta^m$,
so provided $\theta^m<1/(4C_4|k|^{\xi_1+1})$, the $m$-cylinder containing
$x_0$ lies inside $U$.  It suffices to take
$m\approx (\xi_1+1)\ln |k|/(-\ln\theta)$.
By~\eqref{eq-GM},
\[
	\mu_Z(U)\ge\mu_Z(\mathcal{C}_m)\ge C_3^{-1} e^{-g_m(x_0)}\ge 
C_3^{-1} e^{-m|1_{Z_\infty}g|_\infty}\ge C^{-1} |k|^{-(\xi_1+1)\xi_2},
\]
where $\xi_2=|1_{Z_\infty}g|_\infty/(-\ln\theta)$.
Breaking up $Z$ into $U$ and $Z\setminus U$,
\[
|u|_1\le |\hat u|_1\le (1-1/(2|k|^{\xi_1}))\mu_Z(U)+1-\mu_Z(U)=1-\mu_Z(U)/(2|k|^{\xi_1})\le 
1-C^{-1}|k|^{-\xi_3},
\]
where $\xi_3=\xi_1+\xi_2+\xi_1\xi_2$.
By Proposition~\ref{prop-basic},
\begin{align*}
|R_k(\omega)^n u|_\infty & \le |R^n|u||_\infty \le 
|R^n|u|-{\SMALL \int}|u|\,d\mu_Z|_\infty+|u|_1 
 \ll \tau^n\|u\|_\theta + |u|_1 \\
& \le (1+2C_4|k|)\tau^n+1-C^{-1}|k|^{-\xi_3}.
\end{align*}
Choosing $n=n_1(k)=[\zeta_1\ln |k|]$ where $\zeta_1\gg 1$ ensures that
\[
|R_k(\omega)^{\ell(k)+n_1(k)}v|_\infty =
|R_k(\omega)^{n_1(k)}u|_\infty \le 1-C^{-1}|k|^{-\xi_3}.
\]
Setting $n_2(k)=[\zeta_2\ln |k|]$ where $\zeta_2=\zeta+\zeta_1$,
\[
|R_k(\omega)^{n_2(k)}v|_\infty \le 1-C^{-1}|k|^{-\xi_3}.
\]
By Proposition~\ref{prop-infty}(a) and Corollary~\ref{cor-Rk}(b), $|R_k(\omega)^{n_2(k)+n}|_\infty \le 
1-C^{-1}|k|^{-\xi_3}$ for all $n\ge0$, and
\[
|R_k(\omega)^{n_2(k)+n}v|_\theta/(2C_4|k|^\eps) \le {\SMALL\frac 12}+\theta^n C_4\le {\SMALL\frac 34},
\]
for $n$ sufficiently large (independent of $k$).  Increasing $\zeta_2$
slightly, we obtain
$\|R_k(\omega)^{n_2(k)}v\|_k \le 1-C^{-1}|k|^{-\xi_3}$.
Hence $\|(I-R_k(\omega)^{n_2(k)})^{-1}\|_k \le C|k|^{\xi_3}$.
Using the identity $(I-A)^{-1}=(I+A+\dots+A^{m-1})(I-A^m)^{-1}$
and Remark~\ref{rmk-basic}, we obtain
\[
\|(I-R_k(\omega))^{-1}\|_k = O(n_2(k)|k|^{\xi_3})=O(|k|^\xi),
\]
for any choice of $\xi>\xi_3$.
\end{proof}

\begin{lemma} \label{lem-approx}
Assume conditions (iii) and (v).
Then there exists $\xi>0$ and $C\ge1$ such that
$\|(I-R_k(\omega))^{-1}\|_\theta\le C|k|^\xi$ for all $k\in\Z^d\setminus\{0\}$ 
and all $\omega\in[0,2\pi]$.
\end{lemma}

\begin{proof}  This is immediate from Propositions~\ref{prop-approx1}
and~\ref{prop-approx2}.
\end{proof}

\section{Renewal operators}
\label{sec-T}

Define the tower 
$\Delta=\{(z,\ell)\in Z\times\Z: 0\le \ell\le \varphi(z)-1\}$.
The tower map $\hat f:\Delta\to\Delta$ 
is given by 
\[
\hat f(z,\ell)=\begin{cases}  (z,\ell+1), & \ell\le \varphi(z)-2 \\ (Gz,0), & \ell = \varphi(z)-1\end{cases},
\]
with ergodic $\hat f$-invariant measure $\mu_\Delta=\mu_Z\times{\rm counting}$.
Let $L:L^1(\Delta)\to L^1(\Delta)$ denote the transfer operator corresponding to
$\hat f:\Delta\to \Delta$.  (So 
$\int_\Delta Lv\,w\,d\mu_\Delta=\int_\Delta v\,w\circ \hat f\,d\mu_\Delta$.)

Denote by $\pi:\Delta\to X$ the projection $\pi(z,\ell)=f^\ell z$.

\begin{rmk} \label{rmk-mu}
	Since $\pi$ is a semiconjugacy from $\hat f$ to $f$, the measure $\mu=\pi_*\mu_\Delta$ is an ergodic $f$-invariant measure on $X$.   This is the measure described in Section~\ref{sec-GM}.
\end{rmk}

Given a cocycle $h:X\to\T^d$, we define
the lifted cocycle $\hat h=h\circ\pi:\Delta\to\T^d$.
For $k\in\Z^d$, define the twisted transfer operators $L_k:L^1(\Delta)\to L^1(\Delta)$, 
$L_kv=L(e^{ik\cdot \hat h}v)$.

Next, define the renewal operators $T_{k,n}:L^1(Z)\to L^1(Z)$
given by $T_{k,0}=I$ and for $n\ge1$,
\[
T_{k,n}v=1_ZL_k^n(1_Zv)= 1_ZL^n(1_Ze^{ik\cdot \hat h_n}v).
\]
Define $T_k(\omega):L^1(Z)\to L^1(Z)$ for $\omega\in[0,2\pi]$,
\[
T_k(\omega)=\sum_{n=0}^\infty T_{k,n}e^{in\omega}.
\]

Note that $G=\hat f^\varphi:Z\to Z$ is the first return to $Z$ for the map
$\hat f:\Delta\to\Delta$.
Hence for all $k\in\Z^d$ we have the renewal equation,
\[
T_k(\omega)=(I-R_k(\omega))^{-1}.
\]
Let $\hat T_{k,n}$ denote Fourier coefficients of $T_k(\omega)$.

Since the expression $S_{v,w}(n)$ in Theorem~\ref{thm-S1} is a sum over $k\in\Z^d\setminus\{0\}$, we restrict attention throughout to this range of $k$.
(The operators $T_{0,n}$ and $T_0(\omega)$ were studied in~\cite{Gouezel04a,Sarig02,MT12}.)

\begin{prop}  \label{prop-fourier}
Assume condition (iii) and nonexistence of eigenfunctions.  Then
$T_{k,n}=\hat T_{k,n}$ for all $k\in\Z^d\setminus\{0\}$, $n\ge0$.
\end{prop}

For $\beta>1$, this follows from Remark~\ref{rmk-fourier}
using the estimate 
$\hat T_{k,n}=O(n^{-(\beta-\eps)})$,
 and the assumption that there are no eigenfunctions is not required.
(The case $\beta>2$ was treated similarly 
in~\cite{BHM05}.)
The proof of Proposition~\ref{prop-fourier} for general $\beta>0$ is postponed to Appendix~\ref{app-B}.

\begin{lemma} \label{lem-Tk}
Assume conditions (iii) and (v).  
Choose $\eps$ and $\theta$ as in Corollary~\ref{cor-RkHolder}.
Then there are constants $C\ge1$, $\xi>0$, such that
\[
\|T_{k,n}\|\le C|k|^\xi n^{-(\beta-\eps)},
\]
for all $k\in\Z^d\setminus\{0\}$, $n\ge1$.
\end{lemma}

\begin{proof}  
By Corollary~\ref{cor-RkHolder}, $\omega\mapsto R_k(\omega)$ is 
$C^{\beta-\eps}$.
By Lemma~\ref{lem-approx}, $I-R_k(\omega)$ is invertible and so
$\omega\mapsto T_k(\omega)=(I-R_k(\omega))^{-1}$ is $C^{\beta-\eps}$.
Hence by Propositions~\ref{prop-coeff} and~\ref{prop-fourier}, 
\[
\|T_{k,n}\|= \|\hat T_{k,n}\|\ll \|T_k\|_{C^{\beta-\eps}}n^{-(\beta-\eps)}.
\]

By Lemma~\ref{lem-inverse} and Corollary~\ref{cor-RkHolder},
\begin{align*}
\|T_k\|_{C^{\beta-\eps}} & \ll  
\|R_k\|_{C^{\beta-\eps}}^{2\beta+1}\|(I-R_k)^{-1}\|_{C^0}^{2\beta+2}
 \ll |k|^{\eps(2\beta+1)}
\sup_{\omega\in[0,2\pi]}\|(I-R_k(\omega))^{-1}\|^{2\beta+2}.
\end{align*}
Hence by Lemma~\ref{lem-approx}, $\|T_k\|_{C^{\beta-\eps}} \ll |k|^\xi$.
The result follows.
\end{proof}

\section{Proof of Theorem~\ref{thm-S1}}
\label{sec-S1}

In this section, we assume conditions (iii)--(v).
Let $f:X\to X$ 
with induced Gibbs-Markov map $G=f^\varphi:Z\to Z$ as in Section~\ref{sec-GM}.  
Let $\mu_Z$ denote the associated 
ergodic $G$-invariant measure on $Z$.

Let $\hat f:\Delta\to\Delta$ be the tower map defined in Section~\ref{sec-ef}
with
ergodic $\hat f$-invariant measure 
$\mu_\Delta=\mu_Z\times{\rm counting}$.
We continue to let $\pi:\Delta\to X$ denote the semiconjugacy 
$\pi(z,\ell)=f^\ell z$ from $\hat f$ to $f$.  
Recall that $\pi_*\mu_\Delta=\mu$ is the underlying ergodic $f$-invariant measure on $X$.
Given a cocycle $h:X\to\T^d$, we define
the lifted cocycle $\hat h=h\circ\pi:\Delta\to\T^d$.

Fix $\eps\in(0,\beta)$ sufficiently small (to be specified) and $\theta\in[\lambda^{-\eta\eps},1)$.
The symbolic metric $d_\theta$ on $Z$ defined in Section~\ref{sec-GM} extends to a metric on $\Delta$ by defining $d_\theta((z,\ell),(z',\ell'))=\begin{cases} d_\theta(z,z'), & \ell=\ell' \\ 1 & \ell\neq\ell'\end{cases}$.
An observable $v:\Delta\to\R$ is Lipschitz if $\|v\|_\theta=|v|_\infty+|v|_\theta<\infty$ where $|v|_\theta=\sup_{p\neq q}|v(p)-v(q)|/d_\theta(p,q)<\infty$.
Let $F_\theta(\Delta)$ denote the space of Lipschitz observables on $\Delta$.

\begin{prop} \label{prop-hatv}
If $v\in C^\eta(X)$, then $\hat v=v\circ \pi\in F_\theta(\Delta)$.
Moreover, there is a constant $C\ge1$ such that
$\|\hat v\|_\theta\le C\|v\|_{C^\eta}$.
\end{prop}

\begin{proof}
Clearly, $|\hat v|_\infty\le |v|_\infty$.
Let  $q=(z,\ell)$, $q'=(z',\ell')\in \Delta$.
If $\ell\neq\ell'$, we have $|\hat v(q)-\hat v(q')|\le 2|v|_\infty=2|v|_\infty d_\theta(q,q')$.
If $\ell=\ell'$, then setting $C_1'=C_1^\eta$, and using condition~(4) in the definition of nonuniformly expanding map and Proposition~\ref{prop-d},
\begin{align*}
|\hat v(q)-\hat v(q')|=|v(f^\ell z)-v(f^\ell z')
& \le |v|_{C^\eta} d(f^\ell z,f^\ell z')^\eta
\le |v|_{C^\eta} C_1' d(Gz,Gz')^\eta
\\ & \le |v|_{C^\eta} C_1' d_\theta(Gz,Gz')
= |v|_{C^\eta} C_1' \theta^{-1} d_\theta(z,z').
\end{align*}
Hence $|\hat v|_\theta\ll \|v\|_{C^\eta}$.
\end{proof}

In Theorem~\ref{thm-S1}, we are interested in observables $v:X\to\R$ supported in $Y$.  These lift to observables $\hat v:\Delta\to\R$ supported in $\hat Y=\pi^{-1}(Y)$.
Proposition~\ref{prop-hatv} guarantees that if $v\in C^\eta(Y)$, then
$\hat v\in F_\theta(\hat Y)$.


\begin{prop} \label{prop-hatY}
	Let $a\in\alpha$, $0\le \ell<\varphi(a)$.
	If $(a\times\{\ell\})\cap\hat Y\neq\emptyset$, then
	$a\times\{\ell\}\subset\hat Y$.
\end{prop}

\begin{proof}
	Suppose there exists $z_0\in a$ such that $(z_0,\ell)\in\hat Y$.
	Then there exists $q\ge1$ such that $\tau_q(z_0)=\ell$.
	Note that $\tau_q=\ell<\varphi=\tau_\rho$, so $q<\rho$ and $\tau_q$ is constant on $a$ by condition (iv).
	Hence $\tau_q(z)=\ell$ for all $z\in a$, and
	it follows that $a\times\{\ell\}\subset\hat Y$.
\end{proof}

The tower $\Delta$ can be partitioned into levels $\{\Delta_n;\,n\ge0\}$ and
diagonals $\{D_n;\,n\ge1\}$ where
\begin{align*}
	\Delta_n=\{(z,n)\in Z\times\{n\}:\varphi(z)>n\}, \quad
	D_n=\{(z,\varphi(z)-n)\in Z\times\Z:\varphi(z)>n\}.
\end{align*}
Note that $\mu_\Delta(\Delta_n)=\mu_\Delta(D_n)= \mu_Z(\varphi>n)$.
We have the corresponding partitions $\hat Y\cap \Delta_n$ and
$\hat Y\cap D_n$ of $\hat Y$.

\begin{prop} \label{prop-hatY2}  
$\sum_{j\ge n}\mu_\Delta(\hat Y\cap\Delta_j)=O(n^{-(\beta-\eps)})$, 
$\sum_{j\ge n}\mu_\Delta(\hat Y\cap D_j)=O(n^{-(\beta-\eps)})$.
\end{prop}

\begin{proof}
The proof of these estimates is based on~\cite{BruinTerhesiu18}.

First notice that both
$\bigcup_{j\ge n}\hat Y\cap\Delta_j$ and
$\bigcup_{j\ge n}\hat Y \cap D_j$ are contained in
$\{(z,\ell)\in \hat Y:\varphi(z)>n\}$, so it suffices to show that
$\mu_\Delta\{(z,\ell)\in \hat Y:\varphi(z)>n\}=O(n^{-(\beta-\eps)})$.

Next, we write
\[
\SMALL \{(z,\ell)\in \hat Y:\varphi(z)>n\}=\bigcup_{q=1}^\infty
\{(z,\ell)\in \hat Y:\varphi(z)>n,\,\rho(z)=q\}.
\]
If $\rho(z)=q$, then $\varphi(z)=\tau_q(z)$ and so there
are precisely $q$ values of $\ell\in\{0,1,\ldots,\varphi(z)-1\}$ such that $(z,\ell)\in\hat Y$.
Hence
\begin{align*}
\mu_\Delta(\{(z,\ell)\in \hat Y:\varphi(z)>n\}) 
& =\sum_{q=1}^\infty
\mu_\Delta(\{(z,\ell)\in \hat Y:\varphi(z)>n,\,\rho(z)=q\}) \\
& \le \sum_{q=1}^\infty q
\mu_Z(\{z\in Z:\varphi(z)>n,\,\rho(z)=q\}).
\end{align*}
For $k\ge1$,
\begin{align*}
\sum_{q=1}^\infty q \mu_Z(\varphi>n,\,\rho=q) 
& =
\sum_{q=1}^k q \mu_Z(\varphi>n,\,\rho=q) 
+\sum_{q=k+1}^\infty q \mu_Z(\varphi>n,\,\rho=q)  \\
& \le k^2\mu_Z(\varphi>n)+
\sum_{q=k+1}^\infty q \mu_Z(\rho=q) 
\\ & \ll k^2n^{-(\beta-\eps/2)}+ \sum_{q=k+1}^\infty q e^{-cq}
\ll k^2n^{-(\beta-\eps/2)}+ e^{-ck/2},
\end{align*}
where the implied constant is independent of $k$.
Choosing $k=p\log n$ with $p$ sufficiently large, we obtain the desired estimate.
\end{proof}

Recall from Section~\ref{sec-T}
that $L_k:L^1(\Delta)\to L^1(\Delta)$ is the family of twisted transfer operators $L_kv=L(e^{ik\cdot \hat h}v)$ where $\hat h=h\circ\pi$ and $L$ is the transfer operator corresponding to $\hat f$.
From now on, with an obvious abuse of notation, we write $1_{\hat Y}L_k^n1_{\hat Y}$ as a shorthand for
$v\mapsto 1_{\hat Y}L_k^n(1_{\hat Y}v)$.
We view these as operators $1_{\hat Y} L_k^n1_{\hat Y}:F_\theta(\hat Y)\to L^1(\hat Y)$.

Following Gou\"ezel~\cite{GouezelPhD,Gouezel05} (see also~\cite{BHM05}), we define the sequences of operators
\[
A_{k,n}:L^\infty(Z)\to L^1(\Delta), \quad
B_{k,n}:F_\theta(\Delta)\to F_\theta(Z), \quad
E_{k,n}:L^\infty(\Delta)\to L^1(\Delta),
\]
as follows:
\begin{align*}
	(A_{k,n}v)(x) =\!\!\!\!\! \!\!\!\! \sum_{\substack{ f^nz=x \\ z\in Z;\;\hat fz\not\in Z,\ldots,\hat f^nz\not\in Z }}\!\!\!\!\!\!\!\!  & e^{g_n(z)}e^{ik\cdot \hat h_n(z)}v(z), \quad
 (B_{k,n}\hat v)(z) =\!\!\!\!\!\!\! \!\!\!\!\sum_{\substack{ \hat f^nu=z \\ u\not\in Z,\ldots,\hat f^{n-1}u\not\in Z;\;
 \hat f^nu\in Z }}\!\!\!\!\!\!\!\!\!\! e^{g_n(u)}e^{ik\cdot \hat h_n(u)}\hat v(u), \\[.75ex]
  & (E_{k,n}\hat v)(x)  =\!\!\!\! \!\!\!\!\sum_{\substack{ \hat f^nu=x \\ u\not\in Z,\ldots,\hat f^nu\not\in Z }}
 \!\!\!\!\!\!\!\! e^{g_n(u)}e^{ik\cdot \hat h_n(u)}\hat v(u).
\end{align*}
As in~\cite{GouezelPhD,Gouezel05,BHM05},
\begin{align} \label{eq-G}
L_k^n=\sum_{n_1+n_2+n_3=n} A_{k,n_1} T_{k,n_2} B_{k,n_3} +E_{k,n},
\end{align}
and so
\begin{align} \label{eq-G2}
1_{\hat Y}L_k^n1_{\hat Y}=\sum_{n_1+n_2+n_3=n} (1_{\hat Y}A_{k,n_1})\ T_{k,n_2} \ (B_{k,n_3}1_{\hat Y}) \ + \ 1_{\hat Y}E_{k,n}1_{\hat Y},
\end{align}
where 
\[
1_{\hat Y}A_{k,n}:L^\infty(Z)\to L^1(\hat Y),\quad
B_{k,n}1_{\hat Y}:F_\theta(\hat Y)\to F_\theta(Z),\quad
1_{\hat Y}E_{k,n}1_{\hat Y}:L^\infty(\hat Y)\to L^1(\hat Y).
\]

\begin{prop}  \label{prop-ABE}
Uniformly in $k\in\Z^d$, $n\ge1$,
\begin{itemize}
	\item[(a)] $\sum_{j\ge n}\|1_{\hat Y}A_{k,j}\|_{L^\infty(Z)\mapsto L^1(\hat Y)}=O(n^{-(\beta-\eps)})$.
	\item[(b)] $\|1_{\hat Y}E_{k,n}1_{\hat Y}\|_{L^\infty(\hat Y)\mapsto L^1(\hat Y)}=O(n^{-(\beta-\eps)})$.
	\item[(c)] $\sum_{j\ge n}\|B_{k,j}1_{\hat Y}\|_{F_\theta(\hat Y)\mapsto F_\theta(Z)}=O(|k|^\eps n^{-(\beta-\eps)})$.
\end{itemize}
\end{prop}

\begin{proof}
(a)
We have
$|1_{\hat Y}A_{k,n}v|_\infty\le |v|_\infty$ and
$\supp 1_{\hat Y}A_{k,n}v\subset \hat Y\cap\Delta_n$.
Hence
$|1_{\hat Y}A_{k,n}v|_1\le \mu_\Delta(\hat Y\cap\Delta_n)|v|_\infty$ and so
$\|1_{\hat Y}A_{k,n}\|_{L^\infty(Z)\mapsto L^1(\hat Y)}\le \mu_\Delta(\hat Y\cap\Delta_n)$.
Part~(a) now follows from Proposition~\ref{prop-hatY2}.

Similarly
$|1_{\hat Y}E_{k,n}1_{\hat Y}\hat v|_\infty\le |\hat v|_\infty$ and
$\supp 1_{\hat Y}E_{k,n}1_{\hat Y}\hat v\subset \bigcup_{\ell>n}\hat Y\cap\Delta_\ell$.
Hence
$\|1_{\hat Y}E_{k,n}1_{\hat Y}\|_{L^\infty(\hat Y)\mapsto L^1(\hat Y)}\le \sum_{\ell>n}\mu_\Delta(\hat Y\cap\Delta_\ell)$, so part~(b) follows from
Proposition~\ref{prop-hatY2}.

Finally,
\begin{align*}
(B_{k,n}1_{\hat Y}\hat v)(z) & =\sum_{a\in\alpha} 1_{\{\varphi(a)>n\}}e^{g(z_a)}e^{ik\cdot \hat h_n(z_a,\varphi(a)-n)}1_{\hat Y}(z_a,\varphi(a)-n)\hat v(z_a,\varphi(a)-n).
\end{align*}
By Proposition~\ref{prop-hatY}, $1_{\hat Y}(z_a,\varphi(a)-n)=1$ if and only if
$a\times\{\varphi(a)-n\}\subset\hat Y$.
Hence
\begin{align*}
(B_{k,n}1_{\hat Y}\hat v)(z) 
& ={\sum}^* e^{g(z_a)}e^{ik\cdot h_n(f^{\varphi(a)-n}z_a)}\hat v(z_a,\varphi(a)-n),
\end{align*}
where ${\sum}^*$ denotes summation over those $a\in\alpha$ such that $a\times\{\varphi(a)-n\}\subset\hat Y\cap D_n$.
By Proposition~\ref{prop-hatY},
\begin{align} \label{eq-hatY}   
{\sum}^*\mu_Z(a)={\sum}^*\mu_\Delta(a\times\{\varphi(a)-n\})=\mu_\Delta(\hat Y\cap D_n). 
\end{align}
Hence by~\eqref{eq-GM}, $|B_{k,n}1_{\hat Y}\hat v|_\infty \le 
C_3|\hat v|_\infty{\sum}^* \mu_Z(a)
\le C_3 |\hat v|_\infty \mu_\Delta(\hat Y\cap D_n)$.

Also, for $z,z'\in Z$, we have that
$(B_{k,n}1_{\hat Y}\hat v)(z) -(B_{k,n}1_{\hat Y}\hat v)(z')=I_1+I_2+I_3$, where
\begin{align*}
I_1 & ={\sum}^* (e^{g(z_a)}-e^{g(z_a')})e^{ik\cdot h_n(f^{\varphi(a)-n}z_a)}\hat v(z_a,\varphi(a)-n), \\
I_2 & ={\sum}^*  e^{g(z_a')}
(e^{ik\cdot h_n(f^{\varphi(a)-n}z_a)}
-e^{ik\cdot h_n(f^{\varphi(a)-n}z_a')})
\hat v(z_a,\varphi(a)-n), \\
I_3 & ={\sum}^*  e^{g(z_a')}e^{ik\cdot h_n(f^{\varphi(a)-n}z_a')}
(\hat v(z_a,\varphi(a)-n)- \hat v(z_a',\varphi(a)-n)).
\end{align*}
By~\eqref{eq-GM} and~\eqref{eq-hatY},
$|I_1|  \le C_3 |\hat v|_\infty\mu_\Delta(\hat Y\cap D_n)\,d_\theta(z,z')$, and
\begin{align*}
|I_3| & \le C_3 |\hat v|_\theta\mu_\Delta(\hat Y\cap D_n)\,d_\theta((z_a,\varphi(a)-n),(z_a',\varphi(a)-n))
\\ & = C_3\theta |\hat v|_\theta\mu_\Delta(\hat Y\cap D_n)\,d_\theta(z,z').
\end{align*}
Let $\gamma=\theta^{1/\eps}$.
As in the proof of Proposition~\ref{prop-DH}, 
\[
|h_n(f^{\varphi(a)-n}z_a)-
h_n(f^{\varphi(a)-n}z_a')|\le
\sum_{\ell=\varphi(a)-n}^{\varphi(a)-1}
|h|_{C^\eta}d(f^\ell z_a,f^\ell z_a')^\eta \ll nd_\gamma(z,z').
\]
Hence using similar arguments as
in the proof of Lemma~\ref{lem-Rk},
\begin{align*}
|e^{ik\cdot h_n(f^{\varphi(a)-n}z_a)}
-e^{ik\cdot h_n(f^{\varphi(a)-n}z_a')}| 
& \le 2|k|^\eps 
|h_n(f^{\varphi(a)-n}z_a)-
h_n(f^{\varphi(a)-n}z_a')|^\eps
\\ & \ll |k|^\eps n^\eps d_\theta(z,z'). 
\end{align*}
It follows that
\[
|I_2|\ll |\hat v|_\infty |k|^\eps n^\eps \mu_\Delta(\hat Y\cap D_n)\, d_\theta(z,z').
\]
Hence $|B_{k,n}1_{\hat Y}\hat v|_\theta\ll |k|^\eps n^\eps \mu_\Delta(\hat Y\cap D_n)\|\hat v\|_\theta$ and so
$\|B_{k,n}1_{\hat Y}\|_{F_\theta(\hat Y)\mapsto F_\theta(Z)}\ll |k|^\eps n^\eps \mu_\Delta(\hat Y\cap D_n)$.
By Proposition~\ref{prop-hatY2},
$\sum_{j\ge n}\|B_{k,j}1_{\hat Y}\|_{F_\theta(\hat Y)\mapsto F_\theta(Z)}=O(|k|^\eps n^{-(\beta-2\eps)})$, yielding part~(c).
\end{proof}

\begin{cor} \label{cor-ABE}  
There exists $C,\,\xi>0$ such that
\[
\|1_{\hat Y}L_k^n1_{\hat Y}\|_{F_\theta(\hat Y)\mapsto L^1(\hat Y)}\le C|k|^\xi n^{-(\beta-\eps)}
\]
for all $k\in\Z^d\setminus\{0\}$, $n\ge1$.
\end{cor}

\begin{proof}
An elementary calculation shows that if $u_n$, $v_n$ are real sequences and
$|u_n|=O(n^{-\gamma})$,
$\sum_{j\ge n}|v_j|=O(n^{-\gamma})$, where $\gamma>0$, then
$|(u\star v)_n| =O(n^{-\gamma})$.
We apply this with $\gamma=\beta-\eps$.

Note that $G=\hat f^\varphi:Z\to Z$ is the first return map to $Z$ for the tower map $\hat f:\Delta\to\Delta$.
Also, the induced cocycle $H:Z\to\R$ is identical starting from $f$ and $h$ or from $\hat f$ and $\hat h$ so we still have nonexistence of approximate eigenfunctions when working in the tower set up.
Hence Lemma~\ref{lem-Tk} applies and we have that
$\|T_{k,n}\|\ll |k|^\xi n^{-(\beta-\eps)}$.

Combining this with the estimates for
$\sum_{j\ge n}1_{\hat Y}A_{k,j}$ and
$\sum_{j\ge n}B_{k,j}1_{\hat Y}$ in Proposition~\ref{prop-ABE}, it follows that
\[
\Bigl\|\sum_{n_1+n_2+n_3=n} (1_{\hat Y}A_{k,n_1})\ T_{k,n_2} \ (B_{k,n_3}1_{\hat Y})\Bigr\|_{F_\theta(\hat Y)\mapsto L^1(\hat Y)} \ll |k|^{\xi+\eps}n^{-(\beta-\eps)}.
\]
Using~\eqref{eq-G2} and the
estimate for $1_{\hat Y}E_{k,n}1_{\hat Y}$
in Proposition~\ref{prop-ABE}, we obtain the desired estimate
for $1_{\hat Y}L_k^n1_{\hat Y}$.
\end{proof}

\begin{pfof}{Theorem~\ref{thm-S1}}
Since $\pi_*\mu_\Delta=\mu$ and $v$ and $w$ are supported in $Y\times\T^d$, for $k\in\Z^d\setminus\{0\}$ and
$n\ge1$,
\begin{align*}
\int_X e^{ik\cdot h_n}v_{-k}\,w_k\circ f^n\,d\mu
=
\int_\Delta e^{ik\cdot \hat h_n}\hat v_{-k}\,\hat w_k\circ \hat f^n\,d\mu_\Delta
=\int_X 1_{\hat Y}L_k^n1_{\hat Y}\hat v_{-k}\,\hat w_k\,d\mu_\Delta.
\end{align*}
Hence
\[
\Bigl|\int_X e^{ik\cdot h_n}v_{-k}\,w_k\circ f^n\,d\mu\Bigr|
\le |1_{\hat Y}L_k^n1_{\hat Y}v_{-k}|_1 |w_k|_\infty
\le \|1_{\hat Y}L_k^n1_{\hat Y}\|\|v_{-k}\|_\theta |w|_\infty.
\]
By Corollary~\ref{cor-ABE},
$\|1_{\hat Y}L_k^n1_{\hat Y}\|\ll |k|^\xi n^{-(\beta-\eps)}$.
By Proposition~\ref{prop-hatv}, $\|v_{-k}\|_\theta\le C\|v_{-k}\|_{C^\eta}$.
It follows from the usual integration by parts argument that $\|v_{-k}\|_{C^\eta}\ll |k|^{-p}\|v\|_{C^{\eta,p}}$.  
Hence
\[
\Bigl|\int_X e^{ik\cdot h_n}v_{-k}\,w_k\circ f^n\,d\mu\Bigr|
\ll |k|^{\xi-p} n^{-(\beta-\eps)}\|v\|_{C^{\eta,p}}|w|_\infty.
\]
Taking $p>\xi+d$, we obtain that
\[
|S_{v,w}(n)|\ll \sum_{k\in\Z^d\setminus\{0\}}  |k|^{\xi-p} n^{-(\beta-\eps)}\|v\|_{C^{\eta,p}}|w|_\infty\ll n^{-(\beta-\eps)}\|v\|_{C^{\eta,p}}|w|_\infty
\]
as required.
\end{pfof}

\section{Proof of Theorem~\ref{thm-S2}}
\label{sec-S2}

Let $f:X\to X$ 
with induced map Gibbs-Markov map $G=f^\varphi:Z\to Z$ as in Section~\ref{sec-GM}.  
Let $\mu_Z$ denote the associated 
ergodic $G$-invariant probability measure on~$Z$.  
We suppose that
$\mu_Z(\varphi>n)=O(n^{-\beta})$ where $\beta>1$.

Again, we fix $\eps\in(0,\beta)$ sufficiently small (to be specified) and $\theta\in[\lambda^{-\eta\eps},1)$.
We assume in particular that $\beta-\eps>1$.

The tower map
$\hat f:\Delta\to\Delta$, invariant measure 
$\mu_\Delta$, and lifted cocycle
$\hat h=h\circ\pi:\Delta\to\T^d$ are all defined as before.
Also we define
$L:L^1(\Delta)\to L^1(\Delta)$ and
$L_k\hat v=L(e^{ik\cdot \hat h}\hat v)$ as before.

The arguments are similar to those in Section~\ref{sec-S1}, the main differences being that we use~\eqref{eq-G} instead of~\eqref{eq-G2} and that the estimates are simpler but weaker.

\begin{prop} \label{prop-ABE2}
There is a constant $C>0$ such that for all $k\in\Z^d\setminus\{0\}$,
$n\ge1$,
\begin{align*}
  \|A_{k,n}\|_{L^\infty(Z)\mapsto L^1(\Delta)}  & \le   \mu(\varphi>n), \\
  \|B_{k,n}\|_{F_\theta(\Delta)\mapsto F_\theta(Z)} & \le C \mu(\varphi>n)
|k|^\eps n^\eps, \\
 \|E_{k,n}\|_{L^\infty(\Delta)\mapsto L^1(\Delta)} & \SMALL \le \sum_{j>n}\mu(\varphi>j).
\end{align*}
\end{prop}

\begin{proof}  
	These estimates are similar to the ones in Proposition~\ref{prop-ABE}. 
\end{proof}

\begin{cor} \label{cor-ABE2}
Assume condition (v).
There exists $C,\,\xi>0$ such that 
 \[
\|L_k^n\|_{F_\theta(\Delta)\mapsto L^1(\Delta)}\le C|k|^\xi n^{-(\beta-1)}
\]
for all $k\in\Z^d\setminus\{0\}$, $n\ge1$,
\end{cor}

\begin{proof}
We estimate the sequences in~\eqref{eq-G}.
As in the proof of Corollary~\ref{cor-ABE},
$\|T_{k,n}\|\ll |k|^\xi n^{-(\beta-\eps)}$.
By Proposition~\ref{prop-ABE2}, the same estimate holds for
$\|A_{k,n}\|$ and 
$\|B_{k,n}\|$.  Since $\beta-\eps>1$, the convolution of these three sequences 
is also $O(|k|^\xi n^{-(\beta-\eps)})$ for some $\xi$.
Finally, by Proposition~\ref{prop-ABE2}, $\|E_{k,n}\|\ll n^{-(\beta-1)}$.
\end{proof}

\begin{pfof}{Theorem~\ref{thm-S2}}
This follows from Corollary~\ref{cor-ABE2} in the same way that
Theorem~\ref{thm-S1} followed from Corollary~\ref{cor-ABE}.
\end{pfof}

 \appendix
\section{Good asymptotics and nonexistence of approximate eigenfunctions}
\label{app-good}

In this appendix, we 
prove nonexistence of approximate eigenfunctions for an open and dense set of smooth toral extensions.  The method is based on the notion of good asymptotics~\cite{FMT05,FMT07}.

Recall that $G:Z\to Z$ is the induced Gibbs-Markov map with induced cocycle $H:Z\to\T^d$.  
Let $p_0\in Z$ be a fixed point for $G$ and let $p_N$ be a sequence of periodic points, $N\ge1$, with $p_N\to p_0$ and $G^Np_N=p_N$.
We assume that the set of periodic orbits $G^jp_N$, $j\ge0$, $N\ge1$ is contained in a finite union $Z_0$ of partition elements.
In a neighborhood of $p_0$, we can lift $H$ to a cocycle with values in $\R^d$.

\begin{defn} \label{def-good}
The sequence of periodic points $p_N$ has {\em good asymptotics} if
\begin{align} \label{eq-good}
H_N(p_N) & =NH(p_0)+\kappa+J_N\gamma^N+o(\gamma^N) \quad\text{as $N\to\infty$},
\\
\nonumber \varphi_N(p_N)  & =N\varphi(p_0)+\kappa',\quad N\ge1,
\end{align}
where $\gamma\in(0,1)$, $\kappa,\,J_N\in\R^d$, $\kappa'\in\Z$ and
the $i$'th coordinate of $J_N$ has the form
$J_{N,i}=E_{N,i}\cos(N\theta_i+\psi_{N,i})$.
 Moreover, $E_{N,i}$ is a bounded 
sequence of real numbers with $\liminf_{N\to\infty}|E_{N,i}|>0$ for each $i$, and either (a) $\theta_i=0$ and $\psi_{N,i}\equiv0$
or (b) $\theta_i\in(0,\pi)$ and $\psi_{N,i}\in (\tilde\theta_i-\pi/12,\tilde\theta_i+\pi/12)$ for some $\tilde\theta_i$.
\end{defn}

\begin{prop}  \label{prop-good}
If $p_N$ has good asymptotics, then 
 there are no approximate eigenfunctions on the finite subsystem $Z_\infty$ corresponding to $Z_0$.
\end{prop}

\begin{proof}
The argument is an adaptation of~\cite[Proof of Theorem~1.6(a)]{FMT07}.
Suppose that there are approximate eigenfunctions $u_j$ on $Z_\infty$, so
$|M_{k_j,\omega_j}^{n_j}u_j-e^{i\chi_j}u_j|=O(|k_j|^{-\xi})$.
We show that
$\liminf_{N\to\infty}|E_{N,i}|=0$ for some $i\in\{1,\dots,d\}$, so that good asymptotics fails.

Since $|M_{k_j,\omega_j}|_\infty=1$, it is immediate that for all $N\ge1$,
\[
|e^{-i\cdot k_j H_{n_jN}}e^{-i\omega_j \varphi_{n_jN}}u_j\circ G^{n_jN}-e^{iN\chi_j}u_j|
=|M_{k_j,\omega_j}^{n_jN}u_j-e^{iN\chi_j}u_j|
=O(N|k_j|^{-\xi}).
\]
Substituting in the periodic points $p_N$, and using the fact that $|u_j|\equiv1$, we obtain
\[
|e^{i(n_jk_j\cdot H_N(p_N)+n_j\omega_j\varphi_N(p_N)+N\chi_j)}-1|=O(N|k_j|^{-\xi}),
\]
and hence
\[
\dist(n_jk_j\cdot H_N(p_N)+n_j\omega_j\varphi_N(p_N)+N\chi_j,2\pi\Z)=O(N|k_j|^{-\xi}).
\]
Similarly, 
\[
\dist(Nn_jk_j\cdot H(p_0)+Nn_j\omega_j\varphi(p_0)+N\chi_j,2\pi\Z)=O(N|k_j|^{-\xi}).
\]
Subtracting these expressions and using~\eqref{eq-good},
\[
\dist(n_jk_j\cdot(\kappa+J_N\gamma^N+o(\gamma^N))+n_j\omega_j\kappa',2\pi\Z)=O(N|k_j|^{-\xi}).
\]

Recall that $n_j=[\zeta\ln|k_j|]$.
Set $N=N(j)=[\rho\ln|k_j|]$.  For large enough $\rho>0$, we have 
$n_jk_jE_{N(j)}\gamma^{N(j)}=O(|k_j|^{-2\xi})$.
It follows that 
$\dist(n_jk_j\cdot\kappa+n_j\omega_j\kappa',2\pi\Z)=O(|k_j|^{-\xi}\ln|k_j|)$ and so
\begin{align} \label{eq-dist}
\dist(n_jk_j\cdot(J_N\gamma^N+o(\gamma^N)),2\pi\Z)=O(N|k_j|^{-\xi})+O(|k_j|^{-\xi}\ln|k_j|).
\end{align}

Let $S=\sup_N |J_N|$ and set $M(j)=[(\ln(n_j|k_j|)+\ln S + \ln
2)/(-\ln\gamma)]+1$. Then $Sn_j|k_j|\gamma^{M(j)}=\frac12\gamma^{\rho_j}$,
with $\rho_j\in (0,1]$.
In particular, $|Sn_j|k_j|\gamma^{M(j)}|\le\frac12$ and so
taking $N=M(j)+m$ with $m\in\N$ fixed, condition~\eqref{eq-dist} implies that
\begin{align*} 
  \lim_{j\to\infty}n_jk_j\cdot J_{M(j)+m}\gamma^{M(j)}=0.
\end{align*}
Moreover, $n_j|k_j|\gamma^{M(j)}\ge \gamma/(2S)$ and it follows that
there exists $i\in\{1,\dots,d\}$ such that
\begin{align*} 
  \lim_{j\to\infty}E_{M(j)+m,i}\cos((M(j)+m)\theta_i+\psi_{{M(j)+m,i}})=0.
\end{align*}
We show that for this $i$, there is a choice of $m\in\N$ for
which $\cos((M(j)+m)\theta_i+\psi_{M(j)+m,i})$ does not converge to $0$ as
$j\to\infty$

Assume for contradiction that for each integer $m\ge 0$
\begin{equation}
  \label{eq.bad-angles}
  \lim_{j\to\infty}(M(j)+m)\theta_i+\psi_{M(j)+m,i}=\pi/2 \mod \pi.
\end{equation}
Recall that if $\theta_i=0$ then $\psi_{N}\equiv0$,
hence~\eqref{eq.bad-angles} fails (with $m=0$).
Otherwise, $\theta_i\in(0,\pi)$ and $|\psi_{N}-\tilde\theta_i|<\pi/12$.
Taking differences of~\eqref{eq.bad-angles} for various values of
$m$ we obtain that $\ell\theta_i \in [-\pi/6, \pi/6] \bmod \pi$ for all
$\ell$, which is impossible.
\end{proof}

  Next, we recall the construction of periodic orbits with good asymptotics
in~\cite{FMT05,FMT07}.  
We assume that $(X,d)$ is a Riemannian manifold.  Let
$Z_1$ and $Z_2$ be two of the partition elements in $Z$ and assume that these are submanifolds of $X$ and that $G|_{Z_j}:Z_j\to Z$ and
$H|_{Z_j}:Z_j\to\T^d$ are $C^r$ for some $r\ge2$.  These are natural assumptions for piecewise $C^r$ dynamical systems $f:X\to X$ and dynamically $C^r$ cocycles $h:X\to\T^d$.  
For instance, the set up includes Examples~\ref{eg-1} and~\ref{eg-2}; the maps are not $C^2$ for $\gamma<1$, but $G|_a$ is $C^\infty$ for all partition elements $a$.
Similarly, $H|_a$ is $C^r$ in these examples provided $h|_{f^ja}$ is $C^r$ for $j=0,\dots,\varphi(a)-1$.

Let $p_0\in Z_1$ be a fixed point for $G$ and choose a transverse homoclinic point $z\in Z_2$.
Following~\cite{FMT05,FMT07}, we construct a sequence of $N$-periodic points $p_N$, $N\ge1$, for $G$ with orbits lying in $Z_0=Z_1\cup Z_2$. The sequence automatically has good asymptotics except that in exceptional cases there may exist $i$ such that $\liminf_{N\to\infty}|E_{N,i}|=0$.  By~\cite{FMT05,FMT07}, the liminfs are positive for a $C^2$ open and $C^r$ dense set
of cocycles.   (The construction in~\cite{FMT05,FMT07} yields the expression for $H$ in~\eqref{eq-good}, and the same argument gives a similar expression for $\varphi$.  This simplifies as in~\eqref{eq-good} since $\varphi$ is integer-valued.)

Combining this construction with Proposition~\ref{prop-good}, it follows that nonexistence of approximate eigenfunctions holds for an open and dense set of smooth toral extensions.

\section{Proof of Proposition~\ref{prop-fourier}}
\label{app-B}

In this appendix, we show that the coefficients $T_{k,n}$ and $\hat T_{k,n}$ of $T_k$ coincide for all $\beta>0$, $k\in\Z^d\setminus\{0\}$, $n\ge0$.
The case $k=0$ was treated in~\cite{MT12} using a
dominated convergence argument on an annulus at the boundary of the unit disk.
Here we use the same strategy, but the details are somewhat different.

Throughout we assume nonexistence of eigenfunctions, and we work with a fixed $k\in\Z^d\setminus\{0\}$.  Also, we fix $\eps\in(0,1]$ such that $\varphi^\eps\in L^1(Z)$.

Let $\D=\{z\in\C:|z|<1\}$ and $\overline{\D}=\{z\in\C:|z|\le1\}$.
First, we extend the definition of $R_k$ to the closed unit disk, setting $R_k(z)=\sum_{n=1}^\infty R_{k,n}z^n$ for all $z\in\overline{\D}$.
Then $R_k(z)v=R(e^{ik\cdot H}z^\varphi v)$.
Note that $R_k(e^{i\omega})$ coincides with the operator previously denoted $R_k(\omega)$.  

\begin{prop} \label{prop-f1}
$\sup_{\omega\in[0,2\pi]}\|(I-R_k(e^{i\omega}))^{-1}\|_\theta<\infty$.
\end{prop}

\begin{proof}
A standard consequence (see for example~\cite{Hennion93}) of  Proposition~\ref{prop-infty}(b) and Corollary~\ref{cor-Rk}(b) is that $R_k(e^{i\omega})$ has essential spectral radius at most $\theta$.
Hence if $1\in\spec R_k(e^{i\omega})$, then there exists 
 a nonzero function $v\in F_\theta(Z)$ such that $R_k(e^{i\omega})v=v$.
A calculation using the fact that $M_{k,\omega}$ is the $L^2$ adjoint of $R_k(e^{i\omega})$ (see for example~\cite[p.~429]{MN04b}) shows that $M_{k,\omega}v=v$ contradicting the assumption that there are no eigenfunctions.

Hence $1\not\in\spec R_k(e^{i\omega})$, and 
so $\|(I-R_k(e^{i\omega}))^{-1}\|_\theta<\infty$, for each $\omega\in[0,2\pi]$.
By Corollary~\ref{cor-RkHolder}, $\omega\mapsto R_k(e^{i\omega})$ is continuous and
the result follows.
\end{proof}

\begin{rmk}  Under the assumption that there are no approximate eigenfunctions, we could bypass Proposition~\ref{prop-f1} and simply quote Lemma~\ref{lem-approx}.
\end{rmk}

The next step is to extend this estimate to an annulus.

\begin{prop} \label{prop-f2}
There exists $C\ge1$ such that
$\|R_k(e^{i\omega})-R_k(\rho e^{i\omega})\|_\theta\le C(1-\rho)^\eps$,
for all $\rho\in[0,1]$, $\omega\in[0,2\pi]$.
\end{prop}

\begin{proof}
Define $S_{\omega,\rho}=R_k(e^{i\omega})-R_k(\rho e^{i\omega})$.
Let $v\in F_\theta(Z)$.  Then
\[
S_{\omega,\rho}v=
R(e^{ik\cdot H}e^{i\omega\varphi}(1-\rho^\varphi))v.
\]
Hence in the usual notation, for $z\in Z$,
\[
(S_{\omega,\rho}v)(z)=
\sum_{a\in\alpha}e^{g(z_a)}e^{ik\cdot H(z_a)}e^{i\omega\varphi(a)}(1-\rho^{\varphi(a)})v(z_a).
\]
By~\eqref{eq-GM},
\[
|S_{\omega,\rho}v|_\infty
\le C_3|v|_\infty\sum_{a\in\alpha} \mu_Z(a) (1-\rho^{\varphi(a)}).
\]
Now $1-\rho^n\le\min\{1,(1-\rho)n\}\le (1-\rho)^\eps n^\eps$.
Hence, 
\[
|S_{\omega,\rho}v|_\infty
\le C_3|v|_\infty\sum_{a\in\alpha} \mu_Z(a) (1-\rho)^\eps\varphi(a)^\eps
= C_3|\varphi^\eps|_1|v|_\infty (1-\rho)^\eps
\ll  |v|_\infty (1-\rho)^\eps.
\]

Next, for $z,z'\in Z$,
\[
|(S_{\omega,\rho}v)(z)- (S_{\omega,\rho}v)(z')|\le
I_1+I_2+I_3,
\]
where
\begin{align*}
I_1 & = 
	\sum_{a\in\alpha}(e^{g(z_a)}-e^{g(z_a')})e^{ik\cdot H(z_a)}e^{i\omega\varphi(a)}(1-\rho^{\varphi(a)})v(z_a), \\
I_2 & = 
	\sum_{a\in\alpha}e^{g(z_a')}(e^{ik\cdot H(z_a)}-e^{ik\cdot H(z_a')})e^{i\omega\varphi(a)}(1-\rho^{\varphi(a)})v(z_a), \\
I_3 & = 
\sum_{a\in\alpha}e^{g(z_a')}e^{ik\cdot H(z_a')}e^{i\omega\varphi(a)}(1-\rho^{\varphi(a)})(v(z_a)-v(z_a')). 
\end{align*}
Using estimates as in the proof of Lemma~\ref{lem-Rk} combined with the argument above for estimating $1-\rho^{\varphi(a)}$, we obtain
\begin{align*}
& |I_1|  \le C_3|\varphi^\eps|_1|v|_\infty  (1-\rho)^\eps\,d_\theta(z,z'), \\
& |I_2|  \le 2C_2C_3|k|^\eps|h|_{C^\eta}^\eps|\varphi^\eps|_1|v|_\infty(1-\rho)^\eps\,d_\theta(z,z'), \\
& |I_3|  \le C_3|\varphi^\eps|_1|v|_\theta(1-\rho)^\eps\,d_\theta(z,z').
\end{align*}
Hence 
$|S_{\omega,\rho}v|_\theta\ll \|v\|_\theta (1-\rho)^\eps$ and the result follows.
\end{proof}

\begin{cor} \label{cor-f}
There exists $\rho_0\in(0,1]$ such that
\[
\sup_{\rho\in[\rho_0,1]}\sup_{\omega\in[0,2\pi]}\|(I-R_k(\rho e^{i\omega}))^{-1}\|_\theta<\infty.
\]
\end{cor}

\begin{proof}
We use the resolvent identity
\begin{align} \label{eq-ri}
(I-R_k(\rho e^{i\omega}))^{-1}=(I-R_k(e^{i\omega}))^{-1}(I+A_{\omega,\rho})^{-1},
\end{align}
where
\[
A_{\omega,\rho}=(R_k(e^{i\omega})-R_k(\rho e^{i\omega})) (I-R_k(e^{i\omega}))^{-1},
\]
By Propositions~\ref{prop-f1} and~\ref{prop-f2},
$\|A_{\omega,\rho}\|_\theta\ll(1-\rho)^\eps$ for all
$\rho\in[0,1]$, $\omega\in[0,2\pi]$.
Hence we can choose $\rho_0$ so that 
$\|A_{\omega,\rho}\|_\theta\le \frac12$ for 
all $\rho\in[\rho_0,1]$, $\omega\in[0,2\pi]$.
It follows that $\|(I+A_{\omega,\rho})^{-1}\|_\theta\le 2$.
The result follows from~\eqref{eq-ri} and Proposition~\ref{prop-f1}.
\end{proof}

Next, we define $T_k(z)=\sum_{n=0}^\infty T_{k,n}z^n$.
Since $|T_{k,n}|_1\le1$ for all $n$, 
the family $T_k(z)$ is analytic on the open unit disk $\D$ when viewed as a family of operators on $L^1(Z)$.  Hence it is certainly analytic as a family of operators from $F_\theta(Z)$ to $L^1(Z)$.

The renewal equation becomes $T_k(z)=(I-R_k(z))^{-1}$ for $z\in\D$.
By Corollary~\ref{cor-f}, we can extend $T_k(z)$ to $\overline{\D}$ as a continuous family of operators from $F_\theta(Z)$ to $L^1(Z)$.

The Fourier coefficients of $T_k:S^1\to L(F_\theta(Z),L^1(Z))$ are given by
$\hat T_{k,n}=(2\pi)^{-1}\int_0^{2\pi}T_k(e^{i\omega})e^{-in\omega}\,d\omega$.
Also the coefficients of the analytic function $T_k:\D\to L(F_\theta(Z),L^1(Z))$ are given by 
$T_{k,n}=(2\pi)^{-1}\int_0^{2\pi}\rho^{-n}T_k(\rho e^{i\omega})e^{-in\omega}\,d\omega$
for any $\rho\in(0,1]$.
By Corollary~\ref{cor-f} and the renewal equation, the integrand  
$I_\rho(\omega)=\rho^{-n}T_k(\rho e^{i\omega})e^{-in\omega}$ satisfies the uniform bound
$\sup_{\rho\in[\rho_0,1]}\sup_{\omega\in[0,2\pi]}\|I_\rho(\omega)\|_{F_\theta(Z)\mapsto L^1(Z)}<\infty$.
Letting $\rho\to1^-$,  it follows from the dominated convergence theorem that $T_{k,n}=\hat T_{k,n}$ as required.

\paragraph{Acknowledgements}

The research of IM was supported in part by a European Advanced Grant {\em StochExtHomog} (ERC AdG 320977).
This research began at the University of Surrey where IM and DT were supported in part by EPSRC Grant EP/F031807/1.
We are grateful to the referees for very helpful comments.



\end{document}